\documentclass[a4paper,twoside]{amsart}
\pdfoutput=1

\usepackage[T1]{fontenc}
\usepackage[utf8]{inputenc}
\usepackage{lmodern}
\usepackage{enumerate}

\usepackage{amsmath}
\usepackage{amsthm}
\usepackage{amssymb}
\usepackage{amscd}
\usepackage{mathrsfs}
\usepackage{bbold}

\usepackage{tikz}
\usetikzlibrary{arrows,backgrounds,matrix,decorations.pathreplacing,shapes}

\usepackage{scrpage2}

\usepackage{thumbpdf}
\usepackage{hypbmsec}
\usepackage{url}
\usepackage[colorlinks,linkcolor=blue,bookmarksopen=true,bookmarksopenlevel=4,final]{hyperref}
\usepackage{aliascnt}
\usepackage{cleveref}

\setheadsepline{0.4pt}
\automark[subsection]{section}
\ihead{A Simplicial Calculus for Local Intersection Numbers}
\makeatletter
\newtheorem*{rep@theorem}{\rep@title}
\newcommand{\newreptheorem}[2]{%
\newenvironment{rep#1}[1]{%
 \def\rep@title{#2 \ref*{##1}}%
 \begin{rep@theorem}}%
 {\end{rep@theorem}}}
\makeatother

\newtheorem{satz}{Theorem}[section]
\newreptheorem{satz}{Theorem}
\newtheorem{thm}{Theorem}[section]

\newaliascnt{lem}{satz}
\newtheorem{lem}[lem]{Lemma}
\aliascntresetthe{lem}

\newaliascnt{kor}{satz}
\newtheorem{kor}[kor]{Corollary}
\aliascntresetthe{kor}

\newaliascnt{prop}{satz}
\newtheorem{prop}[prop]{Proposition}
\aliascntresetthe{prop}

\newtheorem*{fakt}{Fact}

\theoremstyle{definition}
\newaliascnt{defn}{satz}
\newtheorem{defn}[defn]{Definition}
\newreptheorem{defn}{Definition}
\aliascntresetthe{defn}

\newaliascnt{bsp}{satz}
\newtheorem{bsp}[bsp]{Example}
\aliascntresetthe{bsp}

\newaliascnt{algo}{satz}
\newtheorem{algo}[algo]{Algorithm}
\aliascntresetthe{algo}

\theoremstyle{remark}
\newaliascnt{bem}{satz}
\newtheorem{bem}[bem]{Remark}
\aliascntresetthe{bem}

\crefformat{equation}{#2(#1)#3}
\crefformat{satz}{#2theorem~#1#3} 
\crefformat{defn}{#2definition~#1#3} 
\crefformat{lem}{#2lemma~#1#3} 
\crefformat{kor}{#2corollary~#1#3} 
\crefformat{prop}{#2proposition~#1#3} 
\crefformat{chapter}{#2chapter~#1#3}
\crefformat{bem}{#2remark~#1#3}
\crefformat{thm}{#2theorem~#1#3}
\crefformat{bsp}{#2example~#1#3}
\crefformat{appendix}{#2appendix~#1#3}
\crefformat{algo}{#2algorithm~#1#3}
\numberwithin{equation}{section}

\newcommand{\swz}{Shou-Wu Zhang}

\DeclareMathOperator{\spec}{Spec}
\DeclareMathOperator{\Spec}{Spec}
\DeclareMathOperator{\quot}{Quot}

\DeclareMathOperator{\Proj}{Proj}

\DeclareMathOperator{\Bl}{Bl}
\DeclareMathOperator{\Div}{div}
\DeclareMathOperator{\CH}{CH}
\DeclareMathOperator{\RK}{\mathscr{R}}
\DeclareMathOperator{\Red}{Red}

\DeclareMathOperator{\Hom}{Hom}
\DeclareMathOperator{\supp}{supp}

\DeclareMathOperator{\mycolim}{colim}
\newcommand{\colim}{\mathop{\mycolim}}

\DeclareMathOperator{\ldeg}{ldeg}

\DeclareMathOperator{\im}{Im}

\newcommand{\mId}{\mathfrak{m}}

\newcommand{\Oo}{\mathcal{O}}

\newcommand{\IZ}{\mathbb{Z}}
\newcommand{\IQ}{\mathbb{Q}}
\newcommand{\IR}{\mathbb{R}}
\newcommand{\IN}{\mathbb{N}}

\newcommand{\IF}{\mathbb{F}}

\newcommand{\Llb}{\mathcal{L}}
\newcommand{\betr}[1]{\left| #1 \right|}
\newcommand{\norm}[1]{\left|\!\left| #1 \right|\!\right|}

\newcommand{\angles}[1]{\langle#1\rangle}
\newcommand{\multiangles}{\angles{\cdot,\ldots,\cdot}}
\newcommand{\floors}[1]{\lfloor#1\rfloor}
\newcommand{\pr}{\mathrm{pr}}    % Projektion

\newcommand{\KC}{\mathcal{C}}

\newcommand{\cadiv}{\mathrm{CaDiv}}
\newcommand{\sset}{\mathrm{sSet}}

\newcommand{\sms}{\mathcal{C}^0}
\newcommand{\smq}{\mathcal{C}^\infty_\square}
\newcommand{\smd}{\mathcal{C}^\infty_\vartriangle}
\newcommand{\lid}{\mathcal{C}^\mathrm{lin}_\vartriangle}
\newcommand{\ft}{\mathcal{F}}
\newcommand{\Part}{\mathcal{P}}
\newcommand{\charfkt}{\mathbb{1}}
\newcommand{\diag}{\mathcal{D}}
\newcommand{\poset}{\mathrm{Poset}}
\newcommand{\inner}{^{\mathrm{i}}}

\newcommand{\sd}{\mathrm{sd}}
\newcommand{\unt}{\mathrm{sd}}
\newcommand{\Unt}{\mathrm{Sd}}
\newcommand{\einsvek}{(1, \ldots, 1)}
\newcommand{\smallo}[1]{\mathrm{o}(#1)}
\newcommand{\dund}{\Leftrightarrow}

\newcommand{\BIGOP}[1]{\mathop{\mathchoice%
{\raise-0.22em\hbox{\huge $#1$}}%
{\raise-0.05em\hbox{\Large $#1$}}{\hbox{\large $#1$}}{#1}}}

\newcommand{\BIGboxplus}{\mathop{\mathchoice%
{\raise-0.35em\hbox{\huge $\boxplus$}}%
{\raise-0.15em\hbox{\Large $\boxplus$}}{\hbox{\large $\boxplus$}}{\boxplus}}}

\newcommand{\WV}{\mathrm{WV}}

\begin{document}
    \newlength{\drop}

    \title{An Analytic Description of Local Intersection Numbers at Non-Archimedian Places for Products of Semi-Stable Curves}
    \date{\today}
    \author{Johannes Kolb}

    \maketitle
    
    \begin{abstract}
        We generalise a formula of \swz{} \cite[Thm 3.4.2]{zhang},
        which describes local arithmetic intersection numbers of three Cartier divisors
        with support in the special fibre on a a self-product of a semi-stable arithmetic
        surface using elementary analysis. 
        By an approximation argument, Zhang extends his formula to a formula
        for local arithmetic intersection numbers of three adelic metrized line
        bundles on the self-product of a curve with trivial underlying line
        bundle.
        Using the results on intersection theory from \cite{publ1} we generalize
        these results to $d$-fold self-products for arbitrary $d$. 
        For the approximations to converge, we have to assume that $d$ satisfies the vanishing
        condition \cite[4.7]{publ1}, which is true at least for $d\in \{2,3,4,5\}$.
    \end{abstract}

    \tableofcontents

    \section{Introduction} % {{{

\begin{par}
    Let $R$ be a complete discrete valuation ring
    with algebraically closed residue field $k$.
    We denote the quotient field $\quot(R)$ by $K$
    and a uniformizing element with $\pi \in R$.
    Furthermore let $S$ denote the scheme $\spec{R}$ 
    with generic point $\eta$ and special point $s$. 
    Let $X$ be a regular strict semi-stable $S$-scheme. We denote by
    $\cadiv_{X_s}(X)$ the group of Cartier divisors on $X$ with support in the special fibre $X_s$.
    In \cite{publ1} we studied the pairing
    \begin{equation}
        \label{einl-paar}
        \begin{aligned}
            \big(\cadiv_{X_s}(X)\big)^{d+1} &\to \IZ, \\
            (C_0, \ldots, C_d) &\mapsto \ldeg(C_0 \cdot \cdots \cdot C_d).
        \end{aligned}
    \end{equation}
    given by the intersection product and the degree map.
\end{par}
\begin{par}
    Following an idea of Zhang \cite[\S 3]{zhang} we give an analytical description
    of this pairing:
    First note that the group of Cartier divisors $\cadiv_{X_s}(X)$ with support
    in the special fibre coincides with the free abelian group generated 
    by the vertices of $\Gamma(X)$, the reduction graph of $X$. 
    If we endow the graph $\Gamma(X)$ with a metric
    such that each edge has length $1$, we can describe Cartier divisors
    by continuous functions
    $f_C: |\Gamma(X)| \to \IR$, which are affine on each edge 
    and take only values from $\IZ$ on the vertices.
    We denote the set of functions of these type by $\lid(\Gamma(X))$.
\end{par}
\begin{par}
    Thus the intersection pairing induces a bilinear pairing
    \[
        \angles{\cdot,\cdot}: \lid(\Gamma(X)) \times \lid(\Gamma(X)) \to
        \IR,
    \]
    which is uniquely determined by
    \[
        \angles{f_{C_1}, f_{C_2}} = \ldeg(C_1 \cdot C_2).
    \]
    For this pairing we can give an elementary analytic description:
\end{par}
\begin{fakt}
    Let $X$ be a proper regular strict semi-stable curve and
    $f_1, f_2 \in \lid(\Gamma,\IQ)$. 
    Then
    \begin{equation}
        \label{el-paar-einfach}
        \angles{f_1,f_2} = -\int_{\Gamma(X)}(D^1 f_1) (D^1f_2)
    \end{equation}
    holds. 
    Here we endow the edges of $\Gamma(X)$ with an arbitrary orientation 
    and denote by $D^1 f$ the differential of the function $f$ in the direction
    of this orientation.
\end{fakt}
\begin{par}
    Using approximation arguments we may continue \cref{el-paar-einfach} to a pairing
    of piecewise smooth functions given by the same equation (note that the differential $D^1$
    of piecewise smooth functions is defined almost everywhere, therefore we may still integrate).
    The aim of this paper is to give a generalisation of this description 
    for higher-dimensional schemes.
\end{par}
\begin{par}
    We restrict ourselves to $d$-fold self-products of a smooth curve,
    since we can describe a regular strict semi-stable model explicitly
    for these $K$-schemes:
    If $X_\eta$ is a smooth curve over $K$,
    then the semi-stable reduction theorem yields a semi-stable model of $X$ 
    (possibly after base change). 
    The desingularisation of Gross and Schoen
    yields a uniquely defined regular strict semi-stable scheme $W$ of $X^d$ (see \cite[Thm 3.3]{publ1}).
\end{par}
\begin{par}
    We use as generalisation of the reduction graph
    the geometric realization of the simplicial reduction set
    $|\RK(W)|=|\Gamma(X)|^d$; 
    this is a local affine space, which encodes 
    the incidence relations between the
    components of $W_s$.
    Each Cartier divisor $C \in \cadiv_{W_s}(W)$ with support in $W_s$ 
    is a model of the trivial line bundle, induces therefore a metric on the trivial line bundle,
    which corresponds to a piecewise affine function
    \[
        f_C: |\Gamma(X)|^d \to \IR.
    \]
    We denote the set of these piecewise affine functions by $\lid(\Gamma(X)^d)$.
    The intersection pairing \cref{einl-paar} thus
    induces a multi-linear pairing between piecewise affine functions
    \[
        \multiangles: \Big(\lid(\RK(W))\Big)^{d+1} \to \IR,
    \]
    defined by
    $\angles{f_{C_0}, \ldots, f_{C_d}}=\ldeg(C_0 \cdot \cdots \cdot C_d)$.
\end{par}
\begin{par}
    We see this pairing as the local contribution to the intersection product of metrized line bundles 
    with underlying trivial line bundeles. 
    We want to extend this pairing to a larger class of metrized line bundles.
    By approximation we may continue $\multiangles$ 
    on the set of piecewise smooth functions, $\smd(\Gamma(X)^d)$
    and give an analytical formula for this pairing,
    if a certain vanishing condition holds. 
    This vanishing condition only depends on the positive integer $d$ 
    and can be verified explicitly in the cases $d=2,3$. 
    For a piece-wise smooth function $f \in \smd(\Gamma(X)^d)$ let 
    $f^{(1)},f^{(2)}, \ldots$ denote a certain approximation by piece-wise affine functions
    (see \cref{limit-approx-def}). 
\end{par}
\begin{repsatz}{limit-allg}
    If $d \in \IN$ satisfies the vanishing condition of \cref{limit-konv-bed},
    then for all functions
    $f_0, \ldots, f_d \in \smd(\Gamma^d)$ the limit
    \[ \angles{f_0, \ldots, f_d} := 
        \lim_{n \to \infty} \angles{f_0^{(n)}, \ldots, f_n^{(n)}}_{W,n} \]
    exists.
    It can be calculated by
    \begin{equation}
        \label{einl-haupt-eq}
        \angles{f_0, \ldots, f_d} = 
        \sum_{\Part \textrm{ Partition}}
        \frac{1}{2^{d+|\Part|}}
        \sum_{\substack{
                v_0, \ldots, v_d \in \IF_2^d,\\ 
                \sum \alpha(v_i, \Part) = d+|\Part|
            }}
        \ldeg_{I^d}(\prod_{i=0}^d F_{v_i})
        \int_{\Delta_{\Part}} \prod_{i=0}^d D^{v_i}_{\alpha(v_i,\Part)}(f_i).
    \end{equation}
    % XXX: f^{(n)} nicht erl"autert
\end{repsatz}
\begin{par}
    In this equation the terms $D^{v_i}_{\alpha(v_i,\Part)}(f_i)$
    are elementary analytical expressions in the functions $f_i$.
    The coefficients of the integrals $\ldeg_{I^d}\Big(\prod_i F_{v_i}\Big)$ 
    are independent of $X$ and can 
    be calculated using the simplicial calculus of \cite[4.3]{publ1}.
\end{par}
\begin{par}
    The case $d=2$ was already proven by Zhang in \cite[Prop 3.3.1, Prop 3.4.1]{zhang} 
    for a variant of the Gross-Schoen-desingularization.
    Our proof follows essentially the proof of Zhang, but adds %the following
    new ideas to the proof:
    First of all we use the original method of Gross-Schoen \cite{gross} for desingularization. 
    This requires more technical effort for the description of the special fibre,
    but gives a model with a simpler structure.
\end{par}
\begin{par}
    In \cite[Def 4.22, Prop 4.23]{publ1} we developed a localization argument 
    which reduces the computation of intersection to a simple local situation. 
    %we already proofed explicitly a localisation argument 
    %for these intersection numbers,
    %which Zhang uses implicitly. 
    It reduces in fact the computation of intersection numbers to 
    $\left(\spec{R[x_0,x_1]/(x_0 x_1 - \pi)}\right)^d$, a simple standard-scheme 
    which is independent of $X$.
\end{par}
\begin{par}
    For the derivation of \cref{limit-allg} we have to calculate 
    intersection numbers of divisors which correspond to certain vertices of $\Gamma(X)^d$.
    The vertices serve as nodes for the approximation of piecewise
    smooth functions $f_0, \ldots, f_d \in \smd(\Gamma(X)^d)$
    from \cref{einl-haupt-eq}.
    To get the limit in \cref{limit-allg}, Zhang uses an laborious investigation.
    The resulting formula does not contain the intersection numbers calculated in the first place,
    which makes the generalisation difficult.
    We are able to simplify the argument by using a Fourier transform. 
    This explains the terms $\ft(f_i)$, which appear in the calculation.
    An elementary argument allows us to show that these Fourier transforms
    converge to the generalised differentials $D^{x}_y(f)$.
\end{par}
\begin{par}
    Contrary to the method of Zhang we are able to explain the coefficients
    of the integrals as intersection numbers
    $\ldeg(F_{v_0} \cdot \cdots \cdot F_{v_d})$
    of certain divisors $F_v$ in the standard situation $I^d$.
\end{par}
\begin{par}
    Especially in the case $d=3$ the intersection numbers
    of the $F_v$ can be calculated completely,
    so it is possible to give an explicite description
    of the pairing in this case.  
    In order to simplify the exposition in the introduction,
    we restrict to the special case of functions which are smooth on each cube of $\Gamma^3$,
    a set denote by $\smq(\Gamma^3)$.
\end{par}
\begin{repsatz}{limit-spez-3}
    Let $f_0, \ldots, f_3 \in \smq(\Gamma^3)$ be functions smooth on cubes.
    Then the limit of the quadruple pairing $\angles{f_0,f_1,f_2,f_3}$ 
    exists and can be calculated as
    \begin{align*}
        \lim_{n \to \infty} \angles{f_0^{(n)}, \ldots, f_3^{(n)}} 
        & = \int_{\Gamma^3} 
        \sum_{\substack{
                v_0,v_1,v_2,v_3 \in \IF_2 \\
                \{v_0,v_1,v_2,v_3\} \in B
            }}
        \prod_{i=0}^dD^{v_i}_{|v_i|}(f_i),
    \end{align*}
    where the set $B \subset \mathcal{P}(\IF_2^3)$ is defined as follows
    \begin{align*}
        B:= \Big\{
            &\{(1,0,0), (0,1,0), (0,0,1), (1,1,1)\}, \\
            &\{(1,0,0), (0,1,0), (1,0,1), (0,1,1)\}, \\
            &\{(1,0,0), (0,0,1), (1,1,0), (0,1,1)\}, \\
            &\{(0,1,0), (0,0,1), (1,1,0), (1,0,1)\} \Big\}.
    \end{align*}
\end{repsatz}
\begin{par}
    This investigation can be seen as a first step to an analytical description
    of local arithmetic intersection numbers at non-archimedean places
    in Arakelov theory. 
    For this interpretation of \cref{limit-allg} let $(X_\eta)^{\textrm{an}}$ denote
    the Berkovich analytification of $X_\eta$. 
    By a canonical construction the model $X$ yields a skeleton of $(X_\eta)^{\textrm{an}}$
    which coincides with the geometric realization of $\RK(W)$.
    Thus \cref{limit-allg} gives an intersection number for functions on $(X_\eta)^{\textrm{an}}$,
    if these functions are induced by metrics on the trivial line bundle on $X_\eta$.
\end{par}

% }}}

    % Kapitel Metrik
% vi:ai:sts=4:sw=4:tw=90:foldmethod=marker:syntax=tex

\section{Metrised Line Bundles and The Reduction Map}
\label{metr}

\begin{par}
    Let $R$ be a complete discrete valuation ring with uniformizer $\pi$, whose
    residue field $k:=R/(\pi)$ is algebraically closed.
    Let $S:=\Spec{R} = \{\eta, s\}$ be the corresponding spectrum with 
    generic point $\eta$ and special point $s$.
    Let $X$ be a regular strict semi-stable scheme, i.e. a regular scheme $X$, whose 
    generic fibre $X_\eta$ is smooth and whose special fibre is a reduced divisor 
    with strict normal crossings.
    After choosing a total ordering $\leq$ on $X_s^{(0)}$, the components of $X_s$ we can 
    define a directed reduction graph $\Gamma(X)$ and
    a well-defined Gross-Schoen desingularization of $X^d$ by the following algorithm:
\end{par}
\begin{algo}
    \label{desi-prodkomp}
    Let $d \in \IN$ and $X$ be a regular strict semi-stable $S$-curve with total ordering $\leq$
    on $X_s^{(0)}$ and $\Gamma(X)$ be a simplicial set without multiple simplices.
    We denote the product by $W_0 := X^d$. Since the components of $X_s$ are 
    geometrically integral, we can describe the irreducible components of $(W_0)_s$
    as product
    \[
        (W_0)_s^{(0)} = X_s^{(0)} \times \cdots \times X_s^{(0)}.
    \]
    We endow this product $(W_0)_s^{(0)}$ with the lexicographical order and denote the
    elements in ascending order $B_1, \ldots B_k$.
    Now denote by $B'_1$ the irreducible component $B_1$ endowed with the induced reduced
    structure and
    set $W_1:=\Bl_{B'_1}(W_0)$.
    Inductively let $B'_i \subseteq W_i$ be the strict transform of the
    irreducible component $B_i$ endowed with the induced reduced structure
    and set $W_{i+1}:=\Bl_{B'_i}(W_i)$.
    The last scheme in this chain $W_{k}$ is also denoted by
    $W(X,\leq,d):=W_k$.
    These blowups introduce no new components in the special fibre $(W_k)_s$,
    so the lexicographical ordering on $(W_0)_s^{(0)}$ induces also 
    a total ordering on $W(X,\leq,d)$.
\end{algo}
\begin{par}
    We proved in \cite[Thm 3.3]{publ1}:
\end{par}
\begin{thm}
    The scheme $W:=W(X,\leq,d)$ is a regular strict semi-stable $R$-scheme 
    and for the simplicial reduction set the equation
    \[
        \RK(W) = \Gamma(X)^d
    \]
    holds.
\end{thm}
\begin{par}
    This desingularization
    is denoted by $W$ and is a regular strict semi-stable scheme 
    according to the definition of de Jong \cite{deJong}, 
    which means $W_\eta$ is smooth, $W_s$ is regular and the components 
    of $W_s$ intersect proper and with multiplicity $1$
    (for details see \cite[Prop 4.8]{publ1}).
\end{par}
\begin{par}
    We are about to describe the intersection pairing as paring of functions on the
    reduction set.
    The relation between these functions and the Chow group $\CH^1_{W_s}(W)$ 
    is best described using metrics on line bundles.
\end{par}
% \begin{par}
%     Wir betrachten f"ur einen Moment eine "ubliche Situation der Arakelovtheorie. 
%     Sei dazu $X$ ein regul"ares, projektives und flaches $\IZ$-Schema.
%     Dann gibt es nach \cite[III 4.2]{soule} eine Bijektion zwischen 
%     $\widehat{\mathrm{CH}}^1(X)$, den "Aquivalenzklassen arithmetischer Divisoren,
%     und $\widehat{\mathrm{Pic}}(X)$, 
%     der Gruppe von Paaren $(\mathcal{L}, h)$ 
%     aus einem Geradenb"undel $\mathcal{L}$ 
%     und einer glatten hermiteschen konjugationsinvarianten
%     Metrik $h$ auf $\mathcal{L}(\IC)$,
%     wobei $\mathcal{L}(\IC)$ das von $\mathcal{L}$ auf $X(\IC)$ induzierte 
%     Geradenb"undel bezeichnet.
%     In Analogie dazu betrachtet Zhang in \cite[\S1]{zsmpt} 
%     auf projektiven $\IZ$-Schemata
%     Metriken von Geradenb"undeln
%     auch in der Faser einer nichtarchimedischen Stelle von $\IZ$.
% \end{par}
\begin{par}
    Thus we repeat the definition of metrized line bundles on a complete discrete valued field 
    according to \cite{zsmpt}. In particular we have to deal with metrics 
    induced by models of the trivial line bundle. These metrics are in connection 
    with the irreducible components of the special fibre.
    Using metrics we find an alternative description of the reduction map
    $\Red:W(\bar K) \to |\RK(W)|$.
    Eventually this allows us to define a bijection between cycles on $W$ and functions on the special
    fibre.
\end{par}

\subsection{Metrics on Line Bundles} % {{{

\begin{par}
    Let $R$ be a complete discrete valuation ring with algebraically closed residue field
    and $\betr{\cdot}$ a norm on its quotient field $\quot{R}$.
    In most cases we normalise $\betr{\cdot}$ by setting $\betr{\pi}=1/b$ 
    for a fixed basis $b \in \IR_{>0}$.
\end{par}
\begin{par}
    We denote by $\bar K$ an algebraic closure of $K$ and 
    continue the norm $\betr{\cdot}$ on $\bar K$.
    Since $R$ is complete the completion is unique. 
    We denote by $R_{\bar K}$ the ring of integers
    $\{a \in \bar K \mid \betr{a} \leq 1\}$ in $\bar K$.
\end{par}

% \begin{bem}
%     \label{metr-bew-triv}
%     Der Ring $R_{\bar K}$ ist ein Bewertungsring mit Wertegruppe $\IQ$ von Rang 1, 
%     das Spektrum $\bar S:=\Spec{R_{\bar K}}$ besteht daher 
%     nach \cite[Theorem 10.7]{matsumura}
%     aus einem offenen Punkt $\bar \eta$ 
%     und einem abgeschlossenen Punkt $\bar s$.
%     Insbesondere ist jede offene Umgebung von $\bar s$ bereits das ganze Schema
%     und es folgt $\mathrm{Pic}(\bar S) = 0$.
% \end{bem}
% \begin{par}
%     Sei $X_K$ ein $K$-Schema und $L$ ein Geradenb"undel auf $X_K$. 
%     Zu einem geometrischen Punkt $x \in X_K(\bar K)$, also einem $K$-Morphismus 
%     $\Spec{\bar K} \to X_K$, definieren wir die geometrische Faser des Geradenb"undels
%     am Punkt $x$
%     durch die globalen Schnitte der Garbe $x^{*}(L)$. Sie bildet einen 
%     1-dimensionalen $\bar K$-Vektorraum und
%     wird mit 
%     $L(x):= \Gamma(\Spec{\bar K}, x^{*}(L))$ bezeichnet.
% \end{par}
\begin{defn}
    Let $W_K$ be a $K$-scheme and $L$ a line bundle on $W_K$.
    Let $x \in W_K(\bar K)$ be a $\bar K$-rational point.
    The global sections of $x^{*}(L)$ are called 
    \emph{geometric fiber} of $L$ in $x$ and denoted by
    $L(x):= \Gamma(\Spec{\bar K}, x^{*}(L))$.
    A family of morphisms in the geometric fibres
    $(\norm{\cdot}_x: L(x) \to \IR)_{x \in W_K(\bar K)}$ 
    is called \emph{metric on $L$}, if
    the map
    $\norm{\cdot}_x$ is a $(\bar K, \betr{\cdot})$-norm
    for each $x \in W_k(\bar K)$.
\end{defn}
    
\begin{par}
    Important metrics are given by models of the line bundle $L$ (see \cite[(1.1)]{zsmpt}): 
\end{par}
\begin{defn}
    \label{metr-def-modell}
    Let $W$ be a proper $S$-scheme and $L$ a line bundle on $W_\eta$.
    Let $\Llb$ be a line bundle on $W$, which is a model of $L$ 
    by the isomorphism $\varphi: \Llb_{\eta} \to L$.
    Then there is a metric $\norm{\cdot}_x$ on $L$ defined 
    for a geometric point $x \in W(\bar K)$ and an element
    of the geometric fibre $l \in L(x)$ as follows:
    Let $\tilde x: \bar S \to W$ be the unique continuation 
    of $x: \bar K \to W$ by the valuative criterion of properness
    and let $\tilde \varphi$ be the canonical isomorphism induced by $\varphi$:
    \[ 
        \tilde\varphi: \tilde x^*\Llb (\tilde S_\eta) \to x^*L(\Spec{\bar K})
    \]
    Localization gives a canonical injection of the $\bar R$-module
    $\Gamma(\bar S, \Llb)$ into the $\bar K$-module $\Gamma(\bar S_{\bar\eta}, \Llb)$,
    which allows us to identify $\Gamma(\bar S, \Llb)$ with a subset of 
    $\Gamma(\bar S_{\bar \eta}, \Llb)$. We set
    \[
        \norm{l}_x:= \inf_{a \in \bar K^{\times}}\left(|a| \big| a^{-1} \tilde \varphi^{-1}(l) 
            \in \Gamma(\bar S, \Llb)\right).
    \]
\end{defn}

\begin{bem}
    \label{metr-div-modell}
    Let $W$ be a proper $S$-scheme and $D$ a Cartier divisor with support
    $\supp{D} \subseteq W_s$ in the special fibre of $W$.
    Then the line bundle $\Oo_W(D)$ can be regarded as model of the
    trivial line bundle $\Oo_{W_\eta}$ by the isomorphism
    \[
        \varphi_D: \Oo_W(D)\mid_{W_\eta} \xrightarrow{\sim} \Oo_W\mid_{W_\eta},
    \]
    which maps the canonical section $s_D \in \Oo_W(D)(W_\eta)$ onto 
    $1 \in \Oo_W\mid_{W_\eta}$.
\end{bem}

% \begin{bem}
%     \label{metr-div-modell-alt}
%     Sei $X$ ein eigentliches $S$-Schema und $D$ ein Cartierdivisor, der durch eine
%     invertierbare Untergarbe $\Llb \subseteq \mathcal{K}_X$ gegeben ist. 
%     (vgl.  \cite[II Prop 6.13]{hart})
%     Liegt der Tr"ager $\supp{D}$
%     in der speziellen Faser von $X$, so ist das Geradenb"undel $\Llb$ auf $X_\eta$ trivial
%     und die Inklusion $\Llb \subseteq \mathcal{K}_X$ liefert einen kanonischen
%     Isomorphismus $\varphi_D: \Llb_\eta \simeq \Oo_{X_\eta}$. Das Paar 
%     $(\Llb, \varphi_D)$ bildet daher ein Modell des trivialen Geradenb"undels auf
%     $X_\eta$.
% \end{bem}
\begin{par}
    In this special case the metrics are given by the values of the one-section,
    thus by a function. We will show, that this function plays the role of a coordinate function:
\end{par}
\begin{defn}
    \label{metr-koord-fkt}
    Let $\betr{\cdot}$ be a norm on $\bar K$,
    $W$ a proper $S$-scheme and $D$ a Cartier divisor on $W$ with 
    $\supp{D} \subseteq W_s$.
    The model $(\Oo_W(D), \varphi_D)$ of the trivial line bundle
    induced by $D$ yields by \cref{metr-def-modell} a metric
    $\norm{\cdot}_{\Oo_W(D)}$ on $\Oo_{W_{\eta}}$.
    We evaluate this metric on the one section $1 \in \Gamma(W_\eta, \Oo_{W_\eta})$ 
    to get a function
    \[ f^{\betr{\cdot}}_D: W_{\eta}(\bar K) \to \IR, x \mapsto - \log_b(\norm{1}_{\Oo_W(D),x}). \] 
    It is called the
    \emph{tropical coordinate function} of the divisor $D$.
    If the norm is evident by the context we will denote the coordinate function also
    by $f_D$.
\end{defn}

\begin{prop}
    \label{metr-koord-lokal}
    Let $W$ be a proper $S$-scheme and $D$ a Cartier divisor on $W$ with support in $W$.
    Let $x \in W(\bar K)$ be a $\bar K$-rational point and $U \subseteq W$ an open subset,
    in which $x$ specialises and the Cartier divisor $D$ is given by a rational function 
    $f \in \mathcal{K}_W(U)$.
    Then $f\mid_{U_\eta} \in \Oo_W(U_\eta)$ holds and we have
    \[
        f_{D}(x) = -\log_b\left(\betr{x^{\#}\left(f\big|_{U_\eta}\right)}\right),
    \]
    where $x^{\#}$ denotes the canonical map
    $x^{\#}: \Oo_W(U_\eta) \to \bar K$.
\end{prop}

\begin{proof}
    \begin{par}
        Since the support of $D$ is outside of $U_\eta$,
        we have $f\mid_{U_\eta} \in \Oo_W(U_\eta)$.
        Furthermore $D$ is a principal divisor on $U$ and
        therefore allows us to identify $\Oo_U$ with $\Oo_U(D)$.
        Note that this is nevertheless a non-trivial model 
        of $\Oo_{U_\eta}$, 
        since the model morphism 
        $\varphi: \Oo_U(D) \mid_{U_\eta} \to \Oo_{U_\eta}$
        maps the global section $f \in \Gamma(U_\eta, \Oo_U(D))$ 
        onto $1 \in \Gamma(U_\eta, \Oo_{U_\eta})$.
    \end{par}
    \begin{par}
        Denote by $\tilde x: \bar S \to U$ the 
        continuation of $x$ as before.
        Then the morphism 
        $\tilde x^*\Oo_U(D)(U) \to \tilde x^*\Oo_U(D)(U_\eta)$
        is just the injection 
        $\bar R \to \bar K$.
        The section $1 \in \Oo_{\bar S}(\bar S_\eta)$ is mapped 
        by the model morphism on $x^{\#}(f) \in \bar K$.
        We therefore get
        \[
            f_D(x) = -\log_b(\norm{1}_{\Oo_W(D),x})
            = -\log_b(\inf_{a \in K^\times}\{\betr{a} \mid a^{-1}x^{\#}(f) \in R\})
            = -\log_b(\betr{x^{\#}(f)}).
        \]
    \end{par}
\end{proof}

\begin{prop}
    \label{metr-koordfkt-additiv}
    The coordinate function of the trivial Cartier divisors is the zero function.
    Let $D_1, D_2$ be Cartier divisors on $W$ with $\supp{D_i} \subseteq W_s$.
    Then 
    $f_{D_1+D_2} = f_{D_1}+f_{D_2}$
    holds.
\end{prop}
\begin{proof}
    For each point $x \in W(R_{\bar K})$ it suffices to examine 
    a neighbourhood $U$ of $x$ which trivialises $D$. In $U$ the
    claim follows directly from \cref{metr-koord-lokal}.
\end{proof}

\begin{par}
    Coordinate functions are compatible with base change:
\end{par}

\begin{prop}
    \label{metr-koordfkt-funktor}
    Let $W,V$ be proper integral $S$-schemes,
    $g: W \to V$ a dominant morphism and $D$
    a Cartier divisor on $V$.
    Acoording to \cite[Prop 11.48]{goertz} exists a
    well-defined Cartier divisor $g^*D$.
    Then for each point $x \in W(\bar K)$
    the equation
    \[ f_D(g(x)) = f_{\varphi^*(D)}(x). \]
    holds.
\end{prop}
\begin{proof}
    Let $x \in W(\bar K)$ a $\bar K$-rational point and
    $\tilde x: \bar S \to W$ its continuation on $\bar S$.
    We denote its image under $g$ by $\tilde y:= g \circ \tilde x$.
    As above it suffices to proof the claim in an open neighbourhood
    of $\tilde x$. 
    Thus we may assume that the Cartier divisor $D$ is given on $V$ by
    a rational function $f \in \mathcal{K}_V(V)$. 
    By definition $g^* D$ is represented by $g^* f$ and the claim 
    is a consequence of
    $\tilde y^{\#}(f) = \tilde x^{\#}(g^{\#}(f))$ 
    and \cref{metr-koord-lokal}.
\end{proof}

\begin{bsp}
    \begin{par}
        Let $W$ be a proper $S$-scheme $U \subseteq W$ an affine open subset
        with a dominant morphism
        $f: U \to L = \Spec{R[z_0,z_1]/(z_0z_1 - \pi)}$.
        Let $D$ be a Cartier divisor on $W$ with $\supp(D) \subseteq W_s$,
        which coincides on $U$ with $f^*(\Div(z_0))$.
    \end{par}
    \begin{par}
        We can describe a geometric point $x: \Spec{\bar K} \to U$ of $U$ 
        by its coordinates
        $(f \circ x)^\#(z_0), (f \circ x)^\#(z_1)$.
        If $x \in U(\bar K)$ specialises into $U$, we have
        $(f \circ x)^\#(z_0), (f \circ x)^\#(z_1) \in R_{\bar K}$.
        Since the divisor $D$ is given on $U$ by $f^\#(z_0)$,
        \cref{metr-koord-lokal} implies
        \[ 
            f_D(x) = -\log_b(\betr{(f \circ x)^\#(z_0))}).
        \]
        Thus we get the coordinate function $f_D$
        by the valuation of the $z_0$-component of the point $x$.
    \end{par}
\end{bsp}

% }}}

\subsection{The Reduction Map} % {{{
\begin{par}
    Let us now study the reduction map. First
    we recall the definition of the reduction map
    \cite[2.4.2]{rumely} for curves and define
    a natural generalisation on products of curves.
    Then we describe this reduction map
    using tropical coordinates.
\end{par}
\begin{defn}[{\cite[2.4.2]{rumely}}]
    \label{metr-red-def1}
    Let $X$ be a regular strict semi-stable $S$-curve 
    having a reduction set without multiple simplices.
    Let $\kappa$ denote the inductive system 
    \[ \kappa = \{\tilde K_n \mid n \in \IN, K \subseteq \tilde K_n \subseteq \bar K,
            [K_n : K] = n \}.
    \]
    Then there are canonical isomorphisms
    \[
        \varinjlim_{\tilde K_n \in \kappa} |\RK(X_{\Oo_{\tilde K_n}})|
        \simeq
        \varinjlim_{\tilde K_n \in \kappa} |\sd_n\RK(X)|
        \simeq
        |\RK(X)|.
    \]
    Denote by $\Red_{\tilde K_n}: X(\tilde K_n) \to \RK(X_{\Oo_{\tilde K_n}})$ the map
    which maps a point $x \in X(\tilde K_n)$ to the component $C \in \RK(X_{\Oo{\tilde K_n}})$,
    in which $x$ specialises.
    The limit of these maps induces a map
    \[
        \Red: X(\bar K) = \varinjlim X(\tilde K_n) \to \varinjlim |\RK(X_{\Oo_{\tilde K_n}})|
        \simeq |\RK(X)|,
    \]
    which is called \emph{reduction map}.
\end{defn}
\begin{par}
    We generalise the reduction map to products of curves:
\end{par}
\begin{defn}
    \label{metr-red-def}
    Let $X$ be a regular strict semi-stable $S$-curve
    having a reduction set without multiple simplices, $d \in \IN$
    and $W=W(X,<,d)$ the product model constructed in \cref{desi-prodkomp}.
    Then the map $\Red: W(\bar K) \to |\RK(W)|$,
    which makes the diagram
    \[
        \begin{CD}
            W(\bar K) @>\Red >> |\RK(W)| \\
            @V\pr_i VV    @V\pr_i VV    \\
            X(\bar K) @>\Red >> |\RK(X)| 
        \end{CD}
    \]
    for each $i=1, \ldots d$ commutative,
    is called \emph{reduction map}.
\end{defn}
\begin{par}
    We can give an alternative description of the reduction map
    using coordinate functions.
    Let $X$ be a regular strict semi-stable curve having a reduction set
    without multiple simplices, $d \in \IN$
    and $W=W(X,<,d)$ the product model of $X_\eta$.
\end{par}
\begin{par}
    We identify the vertices $C \in \RK(W)_0$ of the reduction set
    with irreducible components of $W_s$. 
    Since $W$ is regular strict semi-stable, each $C \in \RK(W)_0$
    represents a Cartier divisor and by \cref{metr-koord-fkt}
    we get an associated coordinate function,
    which we denote by $f_C: X_\eta(\bar K) \to \IR$.
\end{par}

\begin{satz}
    \label{metr-red-koord}
    Let $X$ be a regular strict semi-stable curve, $d \in \IN$
    and $W=W(X,<,d)$ the product model of $X_\eta^n$.
    Let $x \in W(\bar K)$ be a geometric point. 
    Then the values of the coordinate functions
    $(f_C(x))_{C \in \RK(W)_0}$ yield
    a probability distribution on $\RK(W)_0$
    with support in a simplex of $\RK(W)$.
    They determine a point $p \in |\RK(W)|$,
    which coincides with $\Red(x) \in |\RK(W)|$.
\end{satz}

\begin{par}
    We split the proof in three parts. First let us show
    that $(f_C(x))_{C \in \RK(W)}$ gives a probability distribution:
\end{par}
\begin{prop}
    \label{metr-koordlike}
    Let $W$ be a proper regular strict semi-stable $S$-scheme 
    having a reduction set without multiple simplices.
    Let $x \in W(\bar K)$ be a geometric point.
    Then 
    $\sum_{C \in \RK(W)_0} f_C(x) = 1$
    and
    $f_C(x) \geq 0$ for all $C \in \RK(W)_0$.
    For each $C \in \RK(W)_0$ the relation $f_C(x)>0$ holds, iff
    $x$ specialises into the component $C$.
\end{prop}

\begin{proof}
    \begin{par}
        The special fibre $W_s$ is given by the principal divisor
        $D_\pi = \Div{\pi}$.
        Since $X_s$ is reduced, 
        $\sum_{C \in \RK(X)_0} C = D_\pi$ holds and by \cref{metr-koordfkt-additiv}
        we get
        $\sum_{C \in \RK(X)_0} f_C(x) = f_{D_\pi}(x)$.
        By \cref{metr-koord-lokal} $f_{D_\pi}(p)=1$ for each $p$,
        which implies the first claim.
    \end{par}
    \begin{par}
        For the second claim let 
        $\tilde x \in X(R_{\bar K})$
        be the continuation of $x \in X(\bar K)$. 
        Let $U \subset X$ be an open neighbourhood
        of $\tilde x$, in which the effective Cartier divisor $C$ 
        is trivialised by a section $h \in \Oo_U(U)$. 
        Then \cref{metr-koord-lokal} implies
        $f_{D_C}(x) = -\log_b(\betr{x^\# h}) \geq 0$.
    \end{par}
    \begin{par}
        By definition the point $x$ specialises into the component $C$,
        iff $\tilde x(\bar s)$ is in $C$, this means
        $h_{\tilde x(\bar s)}$ is in the maximum ideal $\mId_{\tilde x(\bar s)}$
        of the local ring $U_{\tilde x(\bar s)}$. Therefore 
        $\tilde x^\#(h)$ lies in the maximum ideal  $\{x \in R_{\bar K} \mid \betr{x}<1\}$
        of $R_{\bar K}$.
        According to \cref{metr-koord-lokal} this is equivalent to $f_C(x) > 0$.
    \end{par}
\end{proof}

\begin{par}
    Using \cref{metr-koordlike} we are able to prove \cref{metr-red-koord}
    for $d=1$:
\end{par}
\begin{prop}
    \label{metr-red-koord1}
    Let $W=X$ be a proper regular strict semi-stable $S$-curve. 
    Let $n \in \IN$ be a natural number,
    $K_n/K$ a finite field extension of degree $n$
    and $R_n:=\Oo_{K_n}$ the ring of integers in $K$.
    Let $x \in X(K_n)$ be a $K_n$-rational point. 
    Then $\Red(x)$ coincides with the point given by
    the coordinate functions $(f_C(x))_{C \in \RK(X)_0}$.
\end{prop}
\begin{proof}
    Denote by $X_n$ the model of \cref{desi-vbw}. 
    By \cref{metr-koordlike} it suffices to consider the components of $X_s$ resp. $(X_n)_s$,
    in which the point $x$ specialises.
    Thus we may assume that $X$ has the form
    $X=L:=\spec{R[x_0,x_1]/(x_0x_1 - \pi)}$.
    Since $\RK(L)$ has only one edge, 
    the reduction set $\RK(X_n) = \sd_n(\RK(L))$ looks like
    \[
        C'_0 - C'_1 - \cdots - C'_n\ .
    \]
    Choose $i \in \{0, \ldots, n\}$ such that $x$ specialises into $C'_i$.
    Then there is a neighbourhood of $x$ of the form
    $U:=\spec{R_n[y_0,y_1]}$ with structure morphism
    \begin{align*}
        \spec{R[x_0,x_1]/(x_0x_1 - \pi)} &\to \spec{R_n[y_0,y_1]/(y_0y_1 - \tilde\pi)},\\
        x_0 \mapsto y_0^iy_1^{i+1}, x_1 &\mapsto y_0^{n-i}y_1^{n-i-1}.
    \end{align*}
    We may assume that $C'_i$ is given on $U$ by $\Div(y_0)$.
    By \cref{metr-koordlike} we get $f_{C'_i}(x) = 1$.
    Using \cref{metr-koordlike} again 
    we deduce $f_{C_0}(x) = \frac{i}{n}$, $f_{C_1}(x) = \frac{n-i}{n}$ 
    and the claim.
\end{proof}
\begin{par}
    The last part is to consider the product situation $W=W(X,<,d)$:
\end{par}
\begin{prop}
    Let $X$ be a proper regular strict semi-stable curve, further $W=W(X,<,d)$
    and $x \in W(\bar K)$.
    Then the point $\Red(x) \in |\RK(W)| = |\RK(X)^d|$ is given by 
    the probability distribution
    $(f_C(x))_{C \in \RK(W)_0}$.
\end{prop}
\begin{proof}
    By definition \cref{metr-red-def} and \cref{metr-red-koord1} it suffices to show 
    that for each $i \in \{1, \ldots d\}$ the diagram
    \[ 
        \begin{CD}
            W @>>> |\RK(W)|\simeq |\RK(X)^d| \\
            @V \pr_i VV @V \pr_i VV \\
            X @>>> |\RK(X)|
        \end{CD} 
    \]
    commutes. Let $i \in \{1, \ldots d\}$. 
    Since $\pr_i$ is a
    dominant morphism between reduced local noetherian schemes, there exists
    a pull-back of Cartier divisors $\pr_i^*$.
    As $W$ and $X$ are regular strict semi-stable we may identify
    each element $C \in \RK(X)_0$ with a Cartier divisor
    and have the equation
    \[
        \pr_i^*(C) = \sum_{C' \in \RK(W)_0, \pr_i(C')=C} C'.
    \]
    Then \cref{metr-koordfkt-additiv} and \cref{metr-koordfkt-funktor} imply
    \[
        f_C(x) = \sum_{C' \in \RK(W)_0, \pr_i(C')} f_{C'}(x)
    \]
    which is exactly the description of $\pr_i: \RK(X)^d \to \RK(X)$ 
    in coordinate functions.
\end{proof}

\begin{par}
    The reduction map allows us to describe vertical Cartier divisors on $X$
    by analytic objects, precisely by functions on $|\RK(X)|$, the geometric realisation
    of the reduction set.
    We formulate this in the following proposition:
\end{par}
\begin{prop}
    \label{metr-div-fkt}
    Let $X$ be a regular strict semi-stable $S$-curve with total ordering $<$ on
    $X_s^{(0)}$. Let $d \in \IN$ and $W=W(X,\leq,d)$ be the product model
    constructed in \cref{desi-prodkomp} and let $D \in \cadiv_{W_s}{W}$ be a Cartier divisor
    with support in $W_s$.
    Then the function $f_D$ factorises through the reduction map.
    We denote the induced map by $\tilde f_D: |\RK(W)| \to \IR$.
    The function $\tilde f_D$ is affine in each simplex of $\RK(W)$
    and therefore uniquely defined by the values on the vertices
    $(f(C))\mid_{C \in \RK(W)_0}$.
    The divisor can be retrieved by
    \[
        D= \sum_{C \in \RK(W)_0} f_D(C)[C].
    \]
\end{prop}
\begin{proof}
    \begin{par}
        The factorisation is trivial:
        If $D$ is the Cartier divisor of one component $C' \in \RK(W)_0$,
        then $f_{C'}$ is itself a coordinate function and therefore affine on each simplex.
        Since each Cartier divisor is a linear combination of divisors of this form,
        the claim is implied by \cref{metr-koordfkt-additiv}.
    \end{par}
    \begin{par}
        Let $D=\sum_{C \in \RK(W)_0} n_{C}[C]$. 
        For each $C' \in \RK(W)_0$ we choose a point $x \in W(\bar K)$, 
        which specialises only into the component $C'$. By 
        \cref{metr-koordfkt-additiv} and \cref{metr-koordlike}
        we get
        \[
            f_D(x) = \sum_{C \in \RK(W)_0} n_{C}f_C(x) = n_{C'} 
        \]
        and the claim is proven.
    \end{par}
\end{proof}

% }}}

\subsection{Morphisms Between Models} % {{{

\begin{par}
    The coordinate functions $f_D$ are useful to construct morphisms between 
    product models of the type described by \cref{desi-prodkomp} 
    and models arising from ramified base-change in the following way:
\end{par}
\begin{satz}[{\cite[Theorem 3.1]{publ1}}]
   \label{desi-vbw}  % resp. desi-vbw-ex
    Let $S:=\spec{R}$ be the spectrum of a complete discrete valuation ring.
    Let $K_n/K$ be a field extension of degree $n \in \IN$ and $R_n$ the 
    ring of integers in $K_n$.
    We denote $S_n:=\spec{R_n}$.
    Let $X$ be a regular strict semi-stable $S$-curve 
    with a total ordering on $X^{(0)}$, whose simplicial reduction set $\Gamma(X)$ has no
    multiple simplices.
    Let $X_n$ be the scheme obtained by
    blowing up $X \times_S S_n$ successively in all singular points, 
    blowing up the resulting scheme successively in all singular points, and so on $n/2$ times.
    Then $X_n$ is a regular strict semi-stable $S_n$ curve
    with
    $(X_n)_{\eta_n} = (X_\eta) \times_{\spec K} \spec K_n$.
    Furthermore there exists a total ordering of $(X_n)^{(0)}$ such
    that there is a canonical isomorphism 
    \[
        \Gamma(X_n) \simeq \unt_n(\Gamma(X)),
    \]
    which maps the simplicial reduction set of $X_n$ to the canonical
    $n$-fold subdivision $\unt_n(\Gamma(X))$ of the simplicial set $\Gamma(X)$
    (see \cref{sk-unt-def}).
\end{satz}
\begin{par}
    The coordinate functions $f_D$ are useful 
    to construct morphisms between models of the type
    described by \cref{desi-prodkomp} and 
\end{par}
%\begin{par}
%    Die durch die Verfahren aus \cref{desi} gebildeten Modelle sind sich "ahnlich genug,
%    dass man Abbildungen zwischen Modellen erhalten kann,
%    wenn die Reduktionsmengen vertr"aglich sind:
%\end{par}
\begin{satz}
    \label{metr-mod-morph}
    Let $K_n$ be an algebraic field extension of degree $n$,
    $R_n$ the ring of integers of a finite field extension $K_n/K$
    and $S_n:=\spec{R_n}$.
    Let $X$ be a regular strict semi-stable $S$-curve and $X_n$ 
    the model of $X_\eta \times_K K_n$ constructed in \cref{desi-vbw}.
    Let $W=W(X,<,d)$ denote the product model of $X_\eta^d$
    and $W_n=W(X_n,<,d)$ the analogous model of $(X_n)_{\eta_n}^d$.
    Then there exists a morphism $\varphi: W_n \to W$.
\end{satz}
\begin{par}
    The proof uses the universal property of blow-up.
    Its basic idea is the following:
\end{par}
\begin{lem}
    \label{metr-mor-caschni}
    Let $W$ be a proper regular strict semi-stable scheme and $D_1,D_2$ two effective
    Cartier divisors with support in $W_s$.
    The scheme theoretic intersection $D_1 \cap D_2$ is a Cartier divisor
    iff the function 
    $\min(f_{D_1}, f_{D_2})$
    is affine on each simplex of $\RK(W)$.
    In this case the Cartier divisor $D:=D_1 \cap D_2$ provides
    \[
        f_D = \min(f_{D_1}, f_{D_2}).
    \]
\end{lem}
\begin{proof}
    \begin{par}
        Assume that the scheme theoretic intersection $D=D_1 \cap D_2$ is a Cartier
        divisor. 
        Then $f_D$ is affine on each simplex and it suffices to show that 
        $f_D = \min(f_{D_1}, f_{D_2})$.
        Let $\tilde x \in W(R_{\bar K})$ be a $R_{\bar K}$-valued point of $X$
        and $U \subseteq W$ a trivialising neighbourhood of the divisors $D,D_1,D_2$.
        Since the divisors are effective, they are given by sections 
        $r,r_1,r_2 \in \Gamma(\Oo_W, U)$.
        The prerequisite $D=D_1 \cap D_2$ implies $(r) = (r_1,r_2)$
        and thus
        \[
            (x^*(r) ) = (x^*(r_1), x^*(r_2)).
        \]
        Since $R_{\bar K}$ is a principal ideal domain, we have
        $\betr{x^*(r)} = \min(\betr{x^*(r_1)}, \betr{x^*(r_2)})$ 
        and the claim results from \cref{metr-koord-lokal}.
    \end{par}
    \begin{par}
        For the converse we may restrict ourself to one simplex. 
        There the proposition implies
        that one function dominates the other; without loss of generality 
        $f_{D_1} \leq f_{D_2}$.
        Using \cref{metr-koord-lokal} the divisor $D_2 - D_1$ is effective 
        and thus $D= D_1 \cap D_2 = D_1$ is again a Cartier divisor.
    \end{par}
\end{proof}
\begin{par}
    Before we can apply \cref{metr-mor-caschni} in the setting of \cref{metr-mod-morph}
    we show the following compatibility of reduction sets:
\end{par}
\begin{lem}
    \label{metr-mod-vertr}
    Let $X, X_n, W, W_n$ be as in \cref{metr-mod-morph}.
    Then the diagram
    \[
        \begin{CD}
            W_n(\bar K) @>\simeq>> W(\bar K) \\
            @V\Red VV @V\Red VV \\
            |\RK(W_n)| @>\tau>> |\RK(W)|
        \end{CD}
    \]
    commutes, where $\tau: |\RK(W_n)| \simeq |\sd_n\RK(W)| \to |\RK(W)|$
    denotes the canonical morphism of the geometric realisation of the $n$-th subdivision
    (see \cref{sk-unt-kanon}).
    Let $f: |\RK(W)| \to \IR$ be a function which is affine on each simplex.
    Then $f \circ \tau$ is affine on each simplex of $\RK(W_n)$.
\end{lem}
\begin{proof}
    \begin{par}
        According to \cref{metr-red-def1} and \cref{metr-red-def}
        the diagram
        \[
            \begin{CD}
                W_n(\bar K) @>>> \prod_i X_n(\bar K) @>>> \prod_i X(\bar K) @>>> W(\bar K) \\
                @VVV @VVV @VVV @VVV \\
                |\RK(W_n)| @>>> \prod_i |\RK(X_n)| @>>> \prod_i |\RK(X)| @>>> |\RK(W)| 
            \end{CD}
        \]
        commutes and it can be shown that the concatenation of the canonical isomorphisms in the second line
        equals $\tau$.
    \end{par}
    \begin{par}
        The second claim is an implication of the definition of the canonical morphism
        $\tau$ (see \cref{sk-unt-kanon-einb}).
    \end{par}
\end{proof}

\begin{proof}[Proof of \cref{metr-mod-morph}]
    \begin{par}
        By definition $W$ is constructed as gradual blow-up of $W_d$:
        \[ W=W^{[N]} \to \cdots \to W^{[1]} \to W^{[0]}.  \]
        Obviously there is a morphism $\varphi^{[0]}: W_n \to W^{[0]}$. 
        We verify the univseral property of the blow-up to 
        get morphisms $\varphi^{[i]}: W_n \to W^{[i]}$.
    \end{par}
    \begin{par}
        Let $\varphi^{[i]}: W_n \to W^{[i]}$ be already constructed. 
        According to \cite[Lemma 3.6 (iii)]{publ1}
        the centre $C$ of the next blow-up $W^{[i+1]} \to W^{[i]}$ is
        given by an intersection of Cartier divisors
        $C=D_1 \cap \cdots \cap D_l$. 
        The component $C$ is a Cartier divisor on $W^{[i+1]}$ and on $W=W^{[N]}$,
        thus the functions $\min(f_{D_1}, \ldots f_{D_l})$ 
        are affine on each simplex of $\RK(W)$ (\cref{metr-mor-caschni}).
    \end{par}
    \begin{par}
        By \cref{metr-mod-vertr} the pull-backs
        under $\varphi^{[i]}$ are also affine on each simplex of $\RK(W)$
        and by \cref{metr-mor-caschni} the intersection
        $(\varphi^{[i]})^{-1}(C) = (\varphi^{[i]})^* D_1 \cap \cdots \cap (\varphi^{[i]})^* D_l$
        is a Cartier divisor.
        The universal property of the blow-up gives the postulated morphism
        $\varphi^{[i+1]}: W_n \to W^{[i+1]}$.
    \end{par}
\end{proof}

% }}}

    % Kapitel "Limiten von Schnittzahlen"
% vi:ai:sts=4:sw=4:tw=90:foldmethod=marker:syntax=tex

\section{Limits of Intersection Numbers} 
\label{limit}

\begin{par}
    In this last section we combine the theory of metrics
    with the results of \cite{publ1}.
    This allows us to describe the localised intersection numbers of vertical divisors
    by methods of analysis on the simplicial reduction set and use this description
    to approximate hermitean metrics on the trivial bundle.
\end{par}
\begin{par}
    Let $R$ be as usual a complete discrete valuation ring
    with algebraically closed residue field
    and $X$ a proper regular strict semi-stable curve on $S=\Spec{R}$
    with a total ordering $<$ on $X_s^{(0)}$. 
    We want to assume that the reduction set has no multiple simplices.
    Let $W:=W(X,<,d)$ be the regular strict semi-stable model of the product $(X_\eta)^d$
    as defined in \cref{desi-prodkomp}.
    By \cref{metr-div-fkt} there is a bijection between Cartier divisors on $W$ with support in the
    special fibre $W_s$ and piecewise affine functions on $|\RK(W)|$.
    Thus we can view the intersection product as $(d+1)$-fold pairing 
    between piecewise affine functions on $|\RK(W)|$.
\end{par}
\begin{par}
    We use a limit argument in the spririt of Zhang (\cite[Sec. 3]{zhang})
    to continue this pairing on piecewise smooth functions on $|\RK(W)|$
    and give an analytic description of this continuation.
\end{par}

\subsection{Analysis on Simplicial Reduction Sets} % {{{
\begin{par}
    We start with the definition of analytical objects on products of graphs.
    Let $\Gamma$ be a finite graph. By functoriality (\cite[sk-kolim-simpl]{publ1}) each edge 
    $\gamma \in \Gamma_1$ induces an embedding $i_{\gamma}:I \to \Gamma$
    of the standard graph $I=\Delta[1]$ into $\Gamma$.
    In the product we get for each $d$~tuple 
    $\gamma=(\gamma_1, \ldots \gamma_d) \in \Gamma_1^d$
    of edges an embedding
    \[
        i_\gamma := i_{\gamma_1} \times \cdots \times i_{\gamma_d}:
        I^d \to \Gamma^d
    \]
    (compare \cref{schnitt-graph-zerleg}).
    By functoriality this induces a morphism
    $(i_\gamma)_*: |I^d| \to |\Gamma^d|$ 
    of the $d$\nobreakdash-dimensional ``standard cube'' $|I|^d$ into $|\Gamma^d|$.
\end{par}
\begin{par}
    We see these embeddings $(i_\gamma)_*$ as charts.
    To ease the definition of analytical terms
    we identify $|I|^d$ with the standard cube $[0,1]^d$ 
    in $\IR^d$.
    Let $(C_0, \ldots C_k) \in (I^d)_k$ be a $k$\nobreakdash-simplex given by its edges 
    $C_0, \ldots C_k \in (I^d)_0$. Then we denote by $[C_0, \ldots C_k]$ the closed
    convex hull of the points $C_0, \ldots C_k$ --- seen as points in $[0,1]^d$.
    A continuous function $f: [C_0, \ldots C_k] \to \IR$ is called smooth,
    if it can be continued to a smooth function on a open neighbourhood 
    $U \supseteq [C_0, \ldots C_k]$.
\end{par}
\begin{defn}
    \mbox{}
    \begin{enumerate}[(i)]
        \item
            A continuous function $f:|I^d| \simeq [0,1]^d \to \IR$ 
            is called
            \begin{enumerate}[(a)]
                \item
                    \emph{smooth in the cubes}, if $f$ is smooth on $[0,1]^d$,
                \item
                    \emph{smooth in the simplices}, if for each $k \in \IN$ 
                    and each $k$\nobreakdash-simplex $(C_0, \ldots C_k) \in (I^d)_k$
                    the restriction $f \mid_{[C_0, \ldots, C_k]}$ 
                    is smooth,
                \item
                    \emph{affine}, if for each $k \in \IN$ and each $k$\nobreakdash-simplex
                    $(C_0, \ldots C_k) \in (I^d)_k$ the restriction
                    $f \mid_{[C_0, \ldots C_k]}$
                    is affine.
            \end{enumerate}
            \begin{par}
                The set of continuous functions on $|I^d|$ is denoted by
                $\sms(I^d)$, the set of functions smooth in the cubes by $\smq(I^d)$,
                the set of functions smooth in the simplices by $\smd(I^d)$
                and the set of affine functions by $\lid(I^d)$.
            \end{par}
        \item
            \begin{par}
                Let $\Gamma$ be a graph. The set of continuous functions 
                $f: |\Gamma^d| \to \IR$ is denoted by $\sms(\Gamma^d)$.
                A function $f \in \sms(\Gamma^d)$ is called smooth in the cubes
                (smooth in the simplices, affine),
                if for each $\gamma \in (\Gamma_1)^d$
                the function 
                $(i_\gamma)^* f = f \circ i_\gamma: [0,1]^d \to \IR$
                is smooth in the cubes (smooth in the simplices, affine).
            \end{par}
            \begin{par}
                The set of these functions is denoted by $\smq(\Gamma^d)$ 
                ($\smd(\Gamma^d)$, $\lid(\gamma^d)$).
            \end{par}
    \end{enumerate}
\end{defn}

\begin{par}
    To define partial derivatives we have to discuss the points on which the functions
    from $\smd(\Gamma^d)$ and $\lid(\Gamma^d)$ have singularities.
    In the standard cube $|I^d|$ these are exactly the points 
    $x=(x_1, \ldots x_n) \in [0,1]^d$
    where two or more coordinates coincide.
    We call sets of this type \emph{generalised diagonals}:
\end{par}
\begin{defn}
    % XXX: m"ussen die simpl nichtdegeneriert sein?
    The points in $|I^d| \setminus \partial |I^d|$ are called \emph{inner points} of $|I^d|$.
    Let $\Gamma$ be an arbitrary graph. Then $x \in |\Gamma^d|$ is called 
    \emph{inner point}, 
    if there is a tuple $\gamma=(\gamma_1, \ldots \gamma_d) \in \Gamma_1^d$
    such that $x = (i_\gamma)_*(x')$ holds with $x' \in |I^d|$ an inner point.
    In this case $\gamma$ and $x'$ is unique.
    The set of inner points is denoted by $|\Gamma^d|\inner$.
\end{defn}
\begin{defn}
    \mbox{}
    \begin{enumerate}[(i)]
        \item
            Let $x \in |I^d|\inner$ be an inner point in $|I^d|\simeq [0,1]^d$, 
            given by its coordinates $(x_1, \ldots, x_d) \in [0,1]^d$. 
            We define a partition 
            $\{1, \ldots, d\} = A_1 \amalg \cdots \amalg A_l$
            such that
            \[ x_i = x_j \dund \exists h: i,j \in A_h \]
            holds.
            This partition is unique and is denoted by 
            \[
                d(x):=\{A_1, \ldots A_l\}.
            \]
        \item
            Let $x \in |\Gamma^d|\inner$ be an inner point in $|\Gamma^d|$.
            Then there exists a unique tuple $\gamma \in \Gamma_1^d$ and a unique point $x' \in |I^d|$
            such that $(i_\gamma)_*(x') = x$.
            We set
            \[
                d(x) := d(x').
            \]
    \end{enumerate}
\end{defn}
\begin{par}
    The partition $d(x)$ of $x \in |I^d|$ holds the information 
    which coordinates of $x$ coincide. For example the points $x \in |I^d|$
    with $d(x) = \{\{1, \ldots, d\}\}$ are exactly the 
    points of the usual diagonal 
    $\{(t, \ldots, t) \mid t \in (0,1)\}$ in $[0,1]^d$.
\end{par}
\begin{defn}
    Let $\Gamma$ be a graph and $\Part=\{A_1, \ldots, A_k\}$ a partition
    of $\{1, \ldots, d\}$.
    We call the subset
    \[ \diag_\Part(\Gamma^d) := \Big\{x \in |\Gamma^d| \Big| d(x) = \Part\Big\} 
        \subseteq |\Gamma^d|\]
    \emph{generalised diagonal} to the partition $\Part$.
\end{defn}
\begin{bem}
    \label{limit-diag-karte}
    % XXX: Zweite Bedingung als Definition??
    Let $\Part = \{A_1, \ldots, A_l\}$ be a partition of $\{1, \ldots, d\}$
    and $I=\{i_1, \ldots, i_l\}$ with 
    $i_1 \in A_1, \ldots, i_l \in A_l$.
    Then 
    \[
        \diag_\Part(I^d) = \Big\{ x=(x_1, \ldots, x_d) \in |I^d| \Big| 
            x_{j_1} = x_{j_2} \iff \exists m \in \{1, \ldots, l\}: j_1,j_2 \in A_m \Big\}
    \] 
    and by the projection $\pr_I$ on the coordinates $i_1, \ldots, i_l$ we get
    a bijection
    \[ 
        \textrm{c}_\Part: \diag_\Part(I^d) \to [0,1]^{l},
        (x_1, \ldots, x_d) \mapsto (x_{i_1}, \ldots, x_{i_l}).
    \]
    The chart $\textrm{c}_\Part$ does not depend on the choice of $i_1, \ldots, i_l$.
\end{bem}
\begin{par}
    Since continuous functions are integrable, 
    we are able to define an integral on $\Gamma^d$ and on generalised diagonals
    $\diag(\Part)$.
\end{par}
\begin{defn}
    \mbox{}
    \begin{enumerate}[(i)]
        \item
            Let $\Gamma=I$ and $f \in \sms(I^d)$ be a continuous function on the standard
            cube $I^d$. Then $|I^d|$ is canonically homeomorphic to $[0,1]^d$
            and we may define the \emph{integral} of $f$ by
            \[
                \int_{i^d} f:= \int_{[0,1]^d} f(x) d\mu,
            \]
            where $\mu$ denotes the Lebesgue-measure on $[0,1]^d$.
            Let $\Part$ be a partition of the set $\{1, \ldots, d\}$.
            Then the integral along the diagonal $\diag_\Part$ 
            is defined by
            \[
                \int_{\diag_\Part(I^d)} f := \int_{[0,1]^{|\Part|}} f \circ
                \textrm{c}_\Part^{-1} d\mu,
            \]
            where $\textrm{c}_\Part$ denotes the chart from \cref{limit-diag-karte}.
        \item
            Let $\Gamma$ be an arbitrary graph, $f \in \sms(\Gamma^d)$ a continuous
            function on $\Gamma^d$ and $\Part$ a partition of $\{1, \ldots, d\}$.
            Then the \emph{integral} of $f$ along $\Gamma^d$ respectively along
            $\diag_\Part$ is defined by
            \begin{align*}
                \int_{\Gamma^d} f &:= \sum_{\gamma \in \Gamma_1^d} 
                \int_{I^d}(f \circ (i_\gamma)_*), \\
                \int_{\diag_\Part(\Gamma^d)} f & := \sum_{\gamma \in \Gamma_1^d}
                \int_{\diag_\Part(I^d)}(f \circ (i_\gamma)_*) .
            \end{align*}
    \end{enumerate}
\end{defn}
\begin{par}
    For the definition of generalised differential operators 
    on functions smooth in simplices we use a kind of discrete Fourier transform:
\end{par}
\begin{par}
    Let $v=(v_1, \ldots, v_d) \in \IF_2^d$ be a vector and denote by $|v|$ 
    the number of non-trivial components, i.e. $|v| = \#\{i \mid v_i \neq 0 \}$.
    We use these vectors to index the vertices of $I^d$.
    Let furthermore be $h \in \IR$. Then denote by $h^v \in \IR^d$ the point
    \[ h^v:= h\cdot ((-1)^{v_1}, \ldots, (-1)^{v_d}) \]
    in $\IR^d$.
    For each continuous function $f \in \sms(I^d)$ and each $x \in |I^d|\inner$ 
    we study the values of $f$ in a cube surrounding $x$:
    \[
        f^{h}_x: \IF_2^d \to \IR, v \mapsto f(x+h^v).
    \]
\end{par}
\begin{defn}
    \label{limit-diskret-fourier}
    \mbox{}
    \begin{enumerate}[(i)]
        \item
            Let $f \in \sms(I^d)$ be a continuous function on $I^d$, 
            $x \in |I^d|\inner$, $v \in \IF_2^d$. Let $h>0$ small enough,
            such that $x+h^w$ is an inner point for all $w \in \IF_2^d$.
            Then we define a function $\Delta_h^vf$ by the discrete Fourier transform
            \begin{align*}
                \Delta_h^v f: |I^d|\inner &\to \IR\\
                x &\mapsto \frac{1}{2^d} \sum_{w \in \IF_2^d}
                (-1)^{\angles{v,w}}f_x^h(w).
            \end{align*}
        \item
            Let $\Gamma$ be a graph without multiple edges and
            $f \in \sms(\Gamma^d)$.
            Let $x \in |\gamma^d|\inner$ an inner point which is given as image
            of a point $x' \in |I^d|\inner$ under $(i_\gamma)_*:|I^d \to |\Gamma^d|$
            for $\gamma \in \Gamma_1^d$.
            Let $v \in \IF_2^d$ and $h>0$ small enough, that $x'+h^w \in |I^d|\inner$
            holds for all $w \in \IF_2^d$.
            Then define
            \[ \Delta_h^vf(x) := \Delta_h^v f \circ (i_\gamma)_*. \]
            Since $x$ is an inner point, $\gamma$ is unique
            and thus this notation is well-defined.
    \end{enumerate}
\end{defn}

\begin{prop}
    \label{limit-smq-konvergenz}
    Let $f \in \sms(\Gamma^d)$ a continuous function, which is smooth in a neighbourhood
    of a inner point $x \in |\Gamma^d|\inner$. Then $\frac{1}{h^{|v|}}\Delta_h^vf(x)$
    converges to the following differential:
    \[
        \lim_{h \to 0} \frac{1}{h^{|v|}}\Delta_h^vf(x) = 
        D^vf(x) :=
        \left(\frac{\partial}{\partial x_1}\right)^{v_1}\cdots
        \left(\frac{\partial}{\partial x_d}\right)^{v_d}
        f(x).
    \]
\end{prop}
\begin{proof}
    \begin{par}
        It suffices to prove the proposition for the standard graph $\Gamma=I$.
        We use the multidimensional Taylor series of $f$ in $x$:
        Since $f$ is smooth in a neighbourhood of $x$, we have
        for each vector $w \in \IF_2^d$, each $h \in \IR$ and $l \in \IN$:
        \begin{align*}
            f(x+ h^w) & = \sum_{\substack{\lambda \in \IN_0^d\\ 0 \leq |\lambda| \leq l}}
            \frac{(h^w)^\lambda}{\lambda!} D^\lambda f(x) + \smallo{\norm{(h^w)}^{l} } \\
            & = \sum_{\substack{\lambda \in \IN_0^d \\ 0\leq|\lambda|\leq l}}
            (-1)^{\angles{w, \lambda \bmod 2}}\frac{h^{|\lambda|}}{\lambda!} D^\lambda
            f(x) + \smallo{h^{l}}.
        \end{align*}
        The terms $\lambda!$, $(h^w)^\lambda$, $|\lambda|$ are understood
        in the usual multi-index notation.
        For the Fourier transform $\Delta_h^vf$ this implies:
        \begin{align*}
            \Delta_h^vf(x) & = \frac{1}{2^d}\sum_{w \in \IF_2^d} (-1)^{\angles{v,w}}f(x+h^w) \\
            & = \frac{1}{2^d}\sum_{w \in \IF_2^d} (-1)^{\angles{v,w}}
            \sum_{0 \leq |\lambda|\leq |v|} (-1)^{\angles{w, \lambda \bmod 2}}
            \frac{h^{|\lambda|}}{\lambda!} D^\lambda f(x) + \smallo{h^{|v|}} \\
            & = \frac{1}{2^d}\sum_{0 \leq |\lambda| \leq |v|}
            \left(\sum_{w \in \IF_2^d} (-1)^{\angles{w, v - (\lambda \bmod 2)}} \right)
            \frac{h^{|\lambda|}}{\lambda!} D^\lambda f(x) + \smallo{h^{|v|}} \\
            & = h^{|v|}D^vf(x) + \smallo{h^{|v|}}.
        \end{align*}
        We conclude
        \[ 
            \lim_{h \to 0} \frac{1}{h^{|v|}} \Delta_h^vf(x) = D^v f(x) = 
            \left(\frac{\partial}{\partial x_1}\right)^{v_1}\cdots
            \left(\frac{\partial}{\partial x_d}\right)^{v_d}
            f(x).
        \] 
    \end{par}
\end{proof}

\begin{par}
    A similar proposition can be stated for functions, which are smooth on simplices. 
    For theses functions, however, we get a weaker convergence result on the generalised
    diagonals.
\end{par}

\begin{prop}
    \label{limit-sms-konvergenz}
    Let $f \in \smd(\Gamma^d)$ a function in $\Gamma^d$, which is smooth on simplices.
    Let $v \in \IF_2^d$ and $\Part=\{A_1, \ldots, A_k\}$ be a partition 
    of $\{1, \ldots, d\}$.
    We set
    \[ \alpha=\alpha(\Part, v) := 
        \#\{i\in \{1, \ldots, k\} \mid \exists a \in A_i: v_a = 1\}. \]
    Then for each point $x \in \diag_\Part(\Gamma^d)$ the limit
    \begin{equation}
        \label{limit-sms-konvergenz-funk}
        \lim_{h \to 0} \frac{1}{h^\alpha}\Delta^v_hf(x) 
    \end{equation}
    exists and is continuous in $x$ on $\diag_\Part(\Gamma^d)$.
\end{prop}
\begin{par}
    The proof is very technical. Of course it is enough to show
    the proposition for the standard graph $\Gamma=I$. 
    In the next lemma we split the sum $\Delta^v_n(f)(x)$ from
    \cref{limit-diskret-fourier} into parts such that
    each part contains only contributions of one simplex.
    Then $f$ can be replaced by a smooth function and the proposition 
    is a consequence of \cref{limit-smq-konvergenz}.
\end{par}
\begin{par}
    Let $J \subseteq \{1, \ldots, d\}$ be a subset and 
    $H := \{1, \ldots, d\} \setminus J$ its complement.
    For the proof we denote by $V_J$ (resp. $V_H$) 
    the subspace of $V:=\IF_2^d$ spanned by
    the base vectors $e_i$ with $i \in J$ ($i \in H$). 
    The projections in the direct sum $V=V_J \oplus V_H$
    are denoted by $\pr_J$ and $\pr_H$.
\end{par}
\begin{lem}
    \label{limit-sms-konv-hilf}
    Let $f \in \smd(I^d)$ be smooth on simplices and $x \in |I^d|$
    with $d(x)=\{P_1, \ldots, P_k\}$.
    Let $J=\{j_1, \ldots, j_k\}$ with $j_1 \in P_1, \ldots, j_k \in P_k$.
    Denote by $H := \{1, \ldots d\} \setminus J$ its complement.
    Then for each $v \in V$ and $v' \in V_H$
    exists an $\epsilon >0$, an open neighbourhood $U \subseteq |I^d|$
    of $x$ and a function $F$ smooth on $U$ such that
    \[ 
        \sum_{v_H \in V_H} (-1)^{\angles{v_H,v'}}\Delta_h^{v+v_H}(f)(x')
        =\sum_{v_H \in V_H}(-1)^{\angles{v_H,v'}}\Delta_h^{v+v_H}(F)(x')
    \]
    holds for all $x' \in U$ with $d(x') = d(x)$ and all $h < \epsilon$.
\end{lem}
\begin{proof}
    We choose $\epsilon >0$ small enough 
    such that $\epsilon < \frac{1}{4} |x_i - x_j|$
    for all $i,j$ with $x_i \neq x_j$ and set 
    $U' = \{x' \in |I^d| \mid |x' - x| \leq \epsilon\}$.
    For each $x' \in U', d(x') = d(x)$ we have
    \begin{align*}
        \sum_{v_H \in V_H} (-1)^{\angles{v_H,v'}}\Delta_h^{v+v_H}(f)(x') 
        & = \sum_{v_H \in V_H} (-1)^{\angles{v_H,v'}}\sum_{w \in \IF_2^d}
        (-1)^{\angles{v+v_H,w}} f^h_{x'}(w) \\
        & = \sum_{\substack{v_H, w_H \in V_H\\w_J \in V_J}}
        (-1)^{\angles{v_H,v'}+\angles{v+v_H, w_H}+\angles{v,w_J}} f^h_{x'}(w_J+w_H) \\
        & = \sum_{\substack{w_H \in V_H \\ w_J \in V_J }}
        2^{|H|}\delta_{v', w_H}(-1)^{\angles{v, w_H}+\angles{v, w_J}}f^h_{x'}(w_J+w_H) \\
        & = 2^{|H|}\sum_{w_J \in V_J}
        (-1)^{\angles{v, w_J+v'}}f(x'+h^{w_J+v'}).
    \end{align*}
    The points occurring in the last term, 
    \[ Q:= \{x' + h^{w_J + v'} \mid w_J \in V_J, h < \epsilon, 
            x' \in U', d(x')=d(x) \}, \]
    all lie in the same simplex:
    For all points $q \in Q$ the coordinates $q=(q_1, \ldots, q_d) \in \IR^d$ 
    satisfy the inequalities
    \begin{align*}
        q_{k_1} &< q_{k_2} \quad \textrm{ if } x_{k_1} < x_{k_2}, \\
        q_j &\leq q_k \quad \textrm{ if } x_j = x_k, j \in H, v_j = 1, \\
        q_j &\geq q_k \quad \textrm{ if } x_j = x_k, j \in H, v_j = 0. \\
    \end{align*}
    The coordinates of the points in $P$ can therefore be simultaneously sorted by one
    permutation $\sigma \in S_d$ and thus all points of $Q$ are in the 
    corresponding simplex $S_\sigma \subseteq |I^d|$ (see \cref{sk-simpl-perm}).
    Since $f$ is smooth on simplices, there is an open neighbourhood
    $U'' \subseteq S_\sigma$ and a smooth continuation $F$ of $f$.
    By choosing $U:=U' \cap U''$ we get the proposition.
\end{proof}
\begin{kor}
    \label{limit-sms-konv-hilf2}
    Let $f \in \sms(I^d)$, $v \in \IF_2^d$ and $x \in |I^d|$.
    Let $J,H \subseteq \{1, \ldots, d\}$ as in \cref{limit-sms-konv-hilf}.
    Then there exists an open neighbourhood $U \subseteq |I^d|$ of $x$
    and smooth functions $(F_{v_H})_{v_H \in V_H}$ on $U$ such that
    \[ \Delta_h^v(f)(x') = \sum_{v_H \in V_H} \Delta_h^{v+v_H}(F_{v_H})(x') \]
    holds for all $x' \in U$ with $d(x')=d(x)$.
\end{kor}
\begin{proof}
    \begin{par}
        According to \cref{limit-sms-konv-hilf} there is an open neighbourhood $U$
        of $x$, an $\epsilon >0$ and for each $w_H \in V_H$ a smooth function $G_{w_H}$
        on $U$ such that
        \begin{align*}
            \sum_{v_H \in V_H}(-1)^{\angles{v_H,w_H}}\Delta_h^{v+v_H}(f) 
            & = \sum_{v_H \in V_H}(-1)^{\angles{v_H,w_H}}\Delta_h^{v+v_H}(G_{w_H})
        \end{align*}
        holds for each $x' \in U$ with $d(x')=d(x)$ and each $h < \epsilon$.
    \end{par}
    \begin{par}
        By an application of the Fourier transform we get
        \begin{align*}
            \Delta_h^v(f)(x') 
            & = 2^{|H|}\sum_{w_H \in V_H} \left(
                \sum_{v_H \in V_H} (-1)^{\angles{v_H,w_H}}\Delta_h^{v+v_H}(f)
            \right)\\
            & = 2^{|H|}\sum_{w_H \in V_H} \left(
                \sum_{v_H \in V_H} (-1)^{\angles{v_H,w_H}}\Delta_h^{v+v_H}(G_{w_H})
            \right)\\
            & = \sum_{v_H \in V_H} \Delta_h^{v+v_H}\left(
                \sum_{w_H \in V_H} 2^{|H|}(-1)^{\angles{v_H,w_H}}G_{w_H}
            \right)\\
            & = \sum_{v_H \in V_H} 
            \Delta_h^{v+v_H}(F_{v_H}) \\
            \intertext{with}
            F_{v_H}&:= \sum_{w_H \in V_H} 2^{|H|}(-1)^{\angles{v_H, w_H}}G_{w_H}.
        \end{align*}
    \end{par}
\end{proof}

\begin{proof}[Proof of \cref{limit-sms-konvergenz}]
    Denote the partition $d(x)$ by $d(x)=\{A_1, \ldots, A_l\}$.
    From each block $A_j$ we choose an element $j_i \in A_i$ 
    with $v_{j_i} = 0$, in case such an element exists.
    Otherwise we choose an arbitrary element.
    The subset $J:=\{j_1, \ldots, j_l\} \subseteq \{1, \ldots,d \}$
    suffices the conditions of \cref{limit-sms-konv-hilf2}.
    Therefore there is an open neighbourhood $U$ of $x$
    and smooth functions $F_{v_H}$ for each $v_H \in V_H$ such that
    \[ \Delta_n^v(f)(x') = \sum_{v_H \in V_H}\Delta_n^{v+v_H}(F_{v_H}) \]
    holds.
    The minimal value of the set $\{|v+v'|, v' \in V_H\}$ is reached only with
    $v'=w:=\pr_J(v)$. Then $|v+w| = \alpha(x,v)$ and by \cref{limit-smq-konvergenz}
    for each $x' \in U$ with $d(x') = d(x)$ we get
    \begin{align*}
        \lim_{n \to \infty}n^{\alpha(x,v)}\Delta_n^v(f)(x) 
        & = \sum_{v' \in V_H}\lim_{n \to \infty}n^{\alpha(x,v)}\Delta_n^{v+v'}(F_{v'})(x') \\
        & = D^v(F_w)(x')
    \end{align*}
    and thus the proposition.
\end{proof}

\begin{par}
    Motivated by \cref{limit-smq-konvergenz} and \cref{limit-sms-konvergenz} 
    we define the generalised differential:
\end{par}

\begin{defn}
    Let $f \in \sms(\Gamma^d)$ be a continuous function on $|\Gamma^d|$.
    If for $x \in |\Gamma^d|$, $v \in \IF_2^d$ and $\alpha \in \IN$ 
    the limit 
    \[ 
        D_\alpha^v(f)(x):= \lim_{h \to 0} \frac{1}{h^{\alpha}}\Delta_h^v(x) 
    \]
    exists, we call $f$ \emph{differentiable} at $x$ to $v$ in degree $\alpha$
    and $D_\alpha^v$ the \emph{generalised differential} of $f$ to $v$ in degree $\alpha$.
\end{defn}

\begin{par}
    By \cref{limit-sms-konvergenz}, for each $f \in \sms(I^d)$, $v \in \IF_2^d$ and each
    point
    $x \in |I^d|\inner$ the generalised differential 
    in degree $\alpha:=\alpha(d(x),v)$ exists.
    For $d=2$ we can give a connection with the term $\delta(f)$ 
    defined by Zhang in \cite[3.4]{zhang}:
\end{par}

\begin{bsp}
    \label{limit-zhang-delta}
    Let $d=2$ and $f \in \sms(I^2)$. 
    By the diagonal 
    $\diag = \diag(\{\{1,2\}\}) = \{(x_1,x_2) \in (0,1) \mid x_1 = x_2)\}$ 
    the square (\cref{square}) is split into two triangles $S^+,S^- \subseteq |I^2|$,
    on each of which the function $f$ is smooth. 
    By $S^+$ we denote the upper triangle, see \cref{square}.
    \begin{center}
        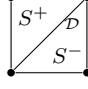
\begin{figure}
        \begin{tikzpicture}
          [place/.style={circle,fill,inner sep=0pt,minimum size=1mm}]
            \node[place] (sw) at (0,0) {};
            \node[place] (se) at (1,0) {};
            \node[place] (nw) at (0,1) {};
            \node[place] (ne) at (1,1) {};
            \node at (0.30,0.75) {\footnotesize$S^+$};
            \node at (0.75,0.25) {\footnotesize$S^-$};
            %\node Diag mitte
            \draw (nw) -- (sw) -- (se) -- (ne) -- (nw);
            \draw (sw) -- (ne);
            \node at (0.80,0.65) {\tiny$\diag$};
        \end{tikzpicture}
        \caption{Square split into two triangles $S^+,S^-$ by diagonal
        $\diag$.}
        \label{square}
    \end{figure}
    \end{center}
    As in Zhang \cite[3.4]{zhang} let $f^+$ and $f^-$ denote
    a continuation of the smooth function $f\mid_{S^+}$ resp. 
    $f\mid_{S^-}$ above the diagonal. 
    For each point $x \in \diag$ and each $0<h \in \IR$ small enough,
    \cref{limit-sms-konv-hilf} implies
    \begin{align*}
        \Delta_h^{(1,1)}(f)(x) + \Delta_h^{(1,0)}(f)(x)
        & = \Delta_h^{(1,1)}(f^+)(x) + \Delta_h^{(1,0)}(f^+)(x), \\
        \Delta_h^{(1,1)}(f)(x) - \Delta_h^{(1,0)}(f)(x)
        & = \Delta_h^{(1,1)}(f^-)(x) - \Delta_h^{(1,0)}(f^-)(x).
    \end{align*}
    We conclude
    \begin{align*}
        \Delta_h^{(1,1)}(f)(x)
        &= \frac{1}{2}\left( \Delta_h^{(1,1)}(f^+)(x) + \Delta_h^{(1,0)}(f^+)(x)
            +  \Delta_h^{(1,1)}(f^-)(x) - \Delta_h^{(1,0)}(f^-)(x)
        \right) \\
        &= \frac{1}{2}\Delta_h^{(1,1)}(f^+ + f^-)(x) +
        \frac{1}{2}\Delta_h^{(1,0)}(f^+-f^-)(x)
    \end{align*}
    and therefore the following limit converges:
    \begin{align*}
        \lim_{h \to 0} \frac{1}{h}\Delta_h^{(1,1)}(f)(x) 
        &= \lim_{h \to 0} \frac{1}{2h} \Delta_h^{(1,0)}(f^+-f^-)(x) \\
        &= \frac{1}{2} \frac{\partial}{\partial x_1} (f^+ - f^-)(x).
    \end{align*}
    Using the notation of Zhang
    $\delta(f):=\frac{\partial}{\partial x_1} (f^+ - f^-)$ 
    this means for each $x \in \diag$:
    \[ D_1^{(1,1)}(f)(x) = \frac{1}{2} \delta(f)(x). \]
\end{bsp}
\begin{par}
    To conclude this section we study a discretization of the 
    generalised differential.
    Let $n \in \IN$. We subdivide the standard cube into cubes
    with edges of length $\frac{1}{n}$. Let $x \in [0,1]^d$ 
    be a point in the standard cube given by its
    coordinates $x=(x_1, \ldots x_d) \in \IR^d$. Let $\floors{x}$ be the vector
    $(\floors{x_1}, \ldots, \floors{x_d})$, where $\floors{\cdot}$ denotes the 
    usual floor function.
    Then the cube of edge length $\frac{1}{n}$ surrounding the point $x$ has 
    the centre coordinates 
    $\tilde x^{(n)}:= \frac{1}{n} \floors{nx} + \frac{1}{2n}\einsvek$.
\end{par}

\begin{defn}
    \mbox{}
    \begin{enumerate}[(i)]
        \item
            Let $f \in \sms(I^d)$ be a continuous function on $[0,1]^d$ and $x \in
            (0,1)^d$. 
            We call the term
            \[ 
                \tilde \Delta_n^vf(x) := \Delta_{1/2n}^v f(\tilde x^{(n)}) =
                \Delta_{1/2n}^v f\Big(\frac{1}{n}
                \floors{nx} + \frac{1}{2n}\einsvek\Big)
            \]
            the \emph{$n$\nobreakdash-th lattice approximation of the derivative to $v \in \IF_2^d$
                at $x$}.
        \item
            Let $\Gamma$ be a graph and $f \in \sms(\Gamma^d)$ a continuous function on
            $|\Gamma^d|$ and $x \in |\Gamma^d|\inner$ an inner point.
            As in \cref{limit-diskret-fourier} there are unique $\gamma \in \Gamma_1^d$,
            $x' \in (I^d)\inner$, such that $(i_\gamma)_*(x')=x$.
            The \emph{$n$\nobreakdash-th lattice approximation of the derivative to $v \in \IF_2^d$
            at $x$} is defined by
            \[ \tilde \Delta_n^vf(x) := 
                \tilde\Delta_n^v( f \circ (i_\gamma)_*)(\tilde x').
            \]
    \end{enumerate}
\end{defn}
\begin{prop}
    \label{limit-tdelta-konv}
    Let $f \in \smd{\Gamma^d}$ be a function smooth on simplices, $x \in |\Gamma^d|\inner$
    an inner point and $\alpha:=\alpha(\Part(x), v)$ as in \cref{limit-sms-konvergenz},
    then
    \[ \lim_{n \to \infty} (2n)^\alpha\tilde\Delta_n^v(x) = D_{v,\alpha}f(x)\]
    holds.
\end{prop}

\begin{proof}
    It suffices to study $I=\Gamma$. Let $N \in \IN$ big enough such that
    the closed line 
    $V:=[ \tilde x^{(N)} - \frac{1}{2N}\einsvek, \tilde x^{(N)} + \frac{1}{2N}\einsvek] $ 
    lies in $I^d$ and $d(x')=d(x)$ holds for all $x \in V$.
    Then for each $x \in V$ and each $n \geq N$ we have
    $\tilde x^{(n)} \in V$. 
    By \cref{limit-sms-konvergenz} there exists
    $D_\alpha^v(f)(x) = \lim_{h \to 0} \frac{1}{h^\alpha}\Delta_h^vf(x)$ 
    for each $x \in V$ and since $V$ is compact this convergence is uniform.
    Again by \cref{limit-sms-konvergenz} the function $D^v_\alpha(f)$ is continuous on $V$
    and thus
    \begin{align*}
        \lim_{n \to \infty} (2n)^\alpha \Delta^v_n(\tilde x^{(n)}) 
        & = \lim_{m \to \infty}\lim_{n \to \infty} (2n)^\alpha \Delta^v_n(\tilde x^{(m)}) \\
        & = \lim_{m \to \infty} D^v_\alpha(f)(\tilde x^{(m)}) = D^v_\alpha(f)(x).
    \end{align*}
\end{proof}

\begin{lem}
    \label{limit-subint-lem}
    Let $n \in \IN$ and $f \in \sms(\Gamma^d)$ be a function with
    $f(x)=f(\tilde x^{(n)})$ for each $x \in |\Gamma^d|$.
    Then for each partition $\Part$ of $\{1, \ldots, d\}$ the equation
    \begin{equation}
        \label{limit-subint-lem-eq}
        \int_{\Gamma^d} f \charfkt_{\{x \mid \Part(\tilde x^{(n)})=\Part\}} = 
        n^{|\Part|-d} \int_{\diag_\Part} f 
    \end{equation}
    holds.
\end{lem}
\begin{proof}
    By definition of the integral it suffices to show this lemma 
    for the standard graph $\Gamma=I$.
    Since $f$ is constant on the set $\{x \in |I|^d \mid \tilde x^{(n)} = y\}$
    for each $y \in \left(\frac{1}{n}\IZ \cap [0,1]\right)^d$,
    we get
    \[ \int_{I^d} f \charfkt_{\{x \mid \Part(\tilde x^{(n)})=\Part\}} 
        = \frac{1}{n^d} \sum_{\substack{
            x \in \left(\frac{1}{n}\IZ \cap[0,1]\right)^d \\
            d(x)=\Part
        }} f(x).
    \]
    Using the chart 
    $c_\Part: \diag(\Part) \to \IR^{|\Part|}$ 
    of the generalised diagonal $\diag(\Part)$ from \cref{limit-diag-karte}
    we conclude
    \begin{align*}
         \int_{I^d} f \charfkt_{\{x \mid \Part(\tilde x^{(n)})=\Part\}}
         &= \frac{1}{n^d} \sum_{
             x \in \left(\frac{1}{n}\IZ \cap[0,1]\right)^{|\Part|}
         } f(c_\Part^{-1}(x)) \\
         &= \frac{1}{n^{d-|\Part|}} \int_{[0,1]^{|\Part|}} f \circ c_\Part^{-1}\mathrm{d}\mu \\
         &= \frac{1}{n^{d-|\Part|}} \int_{\diag(\Part)} f.
    \end{align*}
\end{proof}

\begin{par}
    Let $\sd_n\Gamma$  denote the $n$\nobreakdash-fold subdivision of the graph
    according to \cref{sk-unt-def} and $\sd_n: |\Gamma_n| \to |\Gamma|$ 
    the canonical homeomorphism of the geometric realizations from \cref{sk-unt-kanon}.
    The $d$-th power of this homeomorphism
    yields
    $(\unt_n)^d: |(\Gamma_n)^d| \to |\Gamma^d|$
    and by functoriality there is an homeomorphism
    \[ (\unt_n)^*: \sms(\Gamma^d) \to \sms((\Gamma_n)^d). \]
    This morphism maps all subsets $\sms(\Gamma^d)$ defined so far
    on their counterparts:
\end{par}
\begin{prop}
    \label{schnitt-delta-kommut}
    The homeomorphism $(\unt_n)^*: \sms(\Gamma^d) \to \sms((\Gamma_n)^d)$ satisfies
    \begin{equation}
        \label{limit-unt-lid-eq}
        \begin{split}
            (\unt_n)^* (\smq(\Gamma^d)) &\subseteq \smq(\Gamma_n^d), \\
            (\unt_n)^* (\smd(\Gamma^d)) &\subseteq \smd(\Gamma_n^d), \\
            (\unt_n)^* (\lid(\Gamma^d)) &\subseteq \lid(\Gamma_n^d). \\
        \end{split}
    \end{equation}
    The diagram
    \begin{equation}
        \label{limit-unt-delta-eq}
        \begin{CD}
            \sms(\Gamma^d) @>(\unt_n)^*>> \sms((\Gamma_n)^d) \\
            @V \tilde\Delta_n^vVV     @V\tilde\Delta^v_1VV \\
            \sms(\Gamma^d) @>(\unt_n)^*>> \sms((\Gamma_n)^d) \\
        \end{CD}
    \end{equation}
    commutes and each function $f \in \sms(\Gamma^d)$ 
    satisfies
    \begin{equation}
        \label{limit-unt-int-eq}
        n^d \int_{\Gamma^d} f = \int_{\Gamma_n^d} (\unt_n)^* f.
    \end{equation}
\end{prop}
\begin{proof}
    \begin{par}
        It suffices to proof the claim for $\Gamma=I$.
        The $n$\nobreakdash-fold subdivision of $|I^d|$ is in this case a lattice 
        with distance $1/n$ in $[0,1]^d = |I^d|$.
        Since the image of each simplex of $(I_n)^d$ by $(\unt_n)^*$
        is contained in one simplex of $I^d$, the relations
        \cref{limit-unt-lid-eq} result immediately.
    \end{par}
    \begin{par}
        To proof the commutativity of \cref{limit-unt-delta-eq}
        it suffices to note
	that the edges of $(I_n)^d$ are mapped by $\unt_n$ onto
        the points with rational coordinates
        $[0, 1/n, \ldots, n/n]^d$.
    \end{par}
    \begin{par}
        Equation \cref{limit-unt-int-eq} is an easy implication 
        of integration theory.
    \end{par}
\end{proof}

% }}}

\subsection{The Intersection Pairing} % {{{

\begin{par}
    Let $X$ be a proper regular strict semi-stable $S$\nobreakdash-curve with total ordering $<$ on
    $X^{(0)}$.
    Let its reduction graph $\Gamma(X)$ be without multiple simplices.
    We denote by $W=W(X,<,d)$ the model of $(X_\eta)^d$ constructed in
    \cref{desi-prodkomp}. 
    The norm $\betr{\cdot}$ is chosen such that $|\pi|=1/b$ 
    for a fixed $b \in \IR_{>0}$.
\end{par}
\begin{par}
    Let $\CH^i_{W_s}(W)$ denote the Chow group with codimension $i$ and support $W_s$. 
    By intersection theory there is a intersection product 
    \[
        \cdot: \CH^i_{W_s}(W) \times \CH^j_{W_s}(W) \to \CH^{i+j}_{W_s}(W)_\IQ,
    \]
    and since the special fibre $W_s$ is regular over a field 
    we have also a local degree map
    \[
        \ldeg: \CH^{d+1}_{W_s}(W) \to \IZ.
    \]
    (For details see \cite[Def 4.6]{publ1}).
    We view this local intersection product of divisors with support in $W_s$
    as pairing between affine functions on the reduction set $\lid(\Gamma(X)^d)$
    using the following isomorphism:
\end{par}
\begin{defn}
    Let $X$ be proper regular strict semi-stable curve and $d \geq 2$.
    We denote by $\cadiv_{W_s}(W)_\IR$ the tensor product
    $\cadiv_{W_s}(W) \otimes_{\IZ}\IR$ and by $\phi^{\betr{\cdot}}_1$
    the morphism 
    \[
        \phi^{\betr{\cdot}}_1: \cadiv_{W_s}(W)_\IR \to \lid(\Gamma(X)^d)
    \]
    given by
    \[
        C\otimes r \mapsto r f^{\betr{\cdot}}_C.
    \]
    By \cref{metr-div-fkt} this is an isomorphism.
\end{defn}

\begin{bem}
    Since the intersection product
    \begin{align*}
        \ldeg(\cdot, \ldots, \cdot): \big(\cadiv_{W_s}(W)\big)^{d+1} &\to \IQ, \\
        (D_0, \ldots, D_d) &\mapsto \ldeg(D_0 \cdot \cdots \cdot D_d)
    \end{align*}
    is a multi-linear mapping, we can continue it by linearity
    on $\cadiv_{W_s}(W)_\IR$. We denote this continuation again by $\ldeg$:
    \[
        \ldeg(\cdot, \ldots, \cdot): \big(\cadiv_{W_s}(W)_\IR \big)^{d+1} \to \IR. 
    \]
\end{bem}
\begin{defn}
    \label{limit-schnittpaar-einfach}
    By
    \begin{align*}
        \multiangles_{W,1}: \lid(\Gamma^d)^{d+1} &\to \IR, \\
        (f_0, \ldots, f_d) &\mapsto \ldeg_{W_s}
        ((\phi^{\betr{\cdot}}_1)^{-1}(f_0) \cdots (\phi^{\betr{\cdot}}_1)^{-1}(f_d))
    \end{align*}
    a multi-linear pairing is defined on $\lid(\Gamma^d)$, the set of 
    piecewise affine functions on the reduction set. This pairing is called
    \emph{intersection pairing}.
\end{defn}

\begin{par}
    By the bijection $(\sd_n)^*:\sms(\Gamma^d) \to \sms(\Gamma_n^d)$ 
    of \cref{limit-unt-lid-eq}
    the set $\lid(\Gamma_n^d)$ can be seen as subset of $\sms(\Gamma^d)$
    containing $\lid(\Gamma^d)$.
    We continue the intersection pairing to functions of $\lid(\Gamma_n^d)$.
    Let $K_n/K$ be an algebraic field extension of degree $n$,
    $R_n$ the ring of integers in $K_n$
    and $S_n:=\spec{R_n}$.
    By $X_n$ we denote the regular strict semi-stable model of $X_\eta \times_S S_n$
    defined in \cref{desi-vbw}. 
    Similarly we denote by $W_n:=W(X_n,<, d)$ the model of $(X_\eta)^d \times_S S_n$
    defined in \cref{desi-prodkomp}.
    Since the reduction sets $\RK(X_n)$ and $\RK(W_n)$ are determined combinatorially,
    they are independent of the choice of $K_n$. 
    We denote by $\betr{\cdot}$ also the unique continuation of the valuation
    $\betr{\cdot}: K \to \IR$ to $K_n$.
\end{par}

\begin{par}
    In this situation \cref{metr-div-fkt} yields again 
    an isomorphism of $\IR$\nobreakdash-algebras 
    by \cref{metr-div-fkt},
    \begin{align*}
        \phi^{\betr{\cdot}}_n: \cadiv_{(W_n)_s}(W_n)_\IR &\to \lid((\Gamma(X)_n)^d), 
        \intertext{defined by}
        C\otimes r &\mapsto r f_C^{\betr{\cdot}}. 
    \end{align*}
\end{par}

\begin{defn}
    \label{limit-schnittpaar-bw}
    By
    \begin{align*}
        \multiangles_{W,n}: \lid(\Gamma_n^d)^{d+1} &\to \IR, \\
        (f_0, \ldots, f_d) &\mapsto 1/n \ldeg_{W_s}\Big((\phi^{\betr{\cdot}}_n)^{-1}(f_0) \cdots
        (\phi^{\betr{\cdot}}_n)^{-1}(f_d)\Big)
    \end{align*}
    a multi-linear pairing is defined on $\lid(\Gamma_n^d)$.
    Since local intersection numbers depend only on the structure of $\RK(W_n)$,
    $\multiangles_{W,n}$ is independent on the choice of $K_n$.
\end{defn}
\begin{prop}
    The pairing from \cref{limit-schnittpaar-bw} is a continuation of 
    the pairing from \cref{limit-schnittpaar-einfach}:
    Let
    $f_0, \ldots, f_d \in \lid(\Gamma^d)$ be functions affine on the simplices of
    $\Gamma^d$.
    Then
    \[ \angles{f_0, \ldots, f_d}_{W,n} = \angles{f_0, \ldots, f_d}_{W,1} \]
    holds.
\end{prop}
\begin{proof}
    By linearity we may assume that there are Cartier divisors 
    $D_0', \ldots D_d' \in \cadiv_{W_s}(W)$, 
    such that $f_i = f_{D'_i}^{\betr{\cdot}}$ for each $i$.
    By \cref{metr-mod-morph} there is a morphism $g: W_n \to W$ between the models.
    The pull-back divisors $g^* D_i$ satisfy
    $\phi^{\betr{\cdot}}_n(g^* D_i) = f_i$ by \cref{metr-koordfkt-additiv},
    thus $\angles{f_0, \ldots, f_d}_{W,n} = \ldeg_{W_n}(g^*(D_0) \cdot\cdots\cdot
    g^*(D_d))$.
    As $g: W_n \to W$ is generic flat and 
    $(W_n)_\eta = W_\eta \times_{S_\eta} (S_n)_\eta$ we may apply 
    \cite[schnitt-deg-bw]{publ1}:
    \[ 
        \ldeg_W(D_0 \cdot\cdots\cdot D_d) = n \ldeg_{W_n}(g^*(D_0) \cdot\cdots\cdot g^*(D_d)). 
    \]
    This proofs the proposition.
\end{proof}

\begin{par}
    We may also apply \cref{limit-schnittpaar-einfach} directly 
    to a model $X_n$ on $S_n$.
    The pairing defined in this way is denoted by
    $\multiangles_{W_n,1}$ and coincides with 
    \cref{limit-schnittpaar-bw} up to a constant factor:
\end{par}
\begin{prop}
    \label{limit-schnittpaar-bw-wohldef}
    For $f_0, \ldots, f_d \in \lid(\Gamma_n^d)$ 
    the equation
    \[ \angles{f_0, \ldots, f_d}_{W,n} = 
        n^d \angles{f_0, \ldots, f_d}_{W_n,1} 
    \]
    holds.
\end{prop}
\begin{proof}
    Let $\pi_n$ denote a uniformizer of the discrete valuation ring $R_n$.
    To define $\multiangles_{W_n,1}$ one has to use
    the valuation $\betr{\cdot}_n$ on $R_n$ which is normalised
    by $\betr{\pi_n} = 1/e$.
    Therefore we have $\betr{\cdot}_n = (\betr{\cdot})^n$
    and thus for each Cartier divisor $C \in \cadiv_{(W_n)_s}(W_n)_\IR$:
    \[ f_C^{\betr{\cdot}_n} = nf_C^{\betr{\cdot}}. \]
    We have 
    \[ (\phi_n^{\betr{\cdot}})^{-1}(f_i) = D_i 
        = (\phi_1^{\betr{\cdot}_n})^{-1}(nf_i). \]
    This implies
    \begin{align*}
       \angles{f_0, \ldots, f_d}_{W,n} & = \frac{1}{n}\ldeg_{W_n}(D_0, \ldots, D_d)
       = \frac{1}{n} \angles{nf_0, \ldots, nf_d}_{W_{n, 1}} \\
       & = n^d \angles{f_0, \ldots, f_d}_{W_{n, 1}}. 
    \end{align*}
\end{proof}
\begin{par}
    For our further calculation we need a more explicite description of the Chow ring
    $\CH^1_{W_s}(W)$. We use the combinatorial Chow ring of \cite[Def 4.12]{publ1} defined as:
\end{par}
\begin{defn}
    \label{schnitt-chw}
    \begin{par}
        Let $\Gamma$ be a finite graph without multiple simplices 
        and $\Gamma^d$ be the $d$-fold product.
        We denote with $Z(\Gamma^d)$ the polynomial ring
        $Z(\Gamma^d):= \IZ[C \mid C \in (\Gamma^d)_0]$ 
        generated by the 0-simplices. 
        It is supplied with the usual grading, which
        gives all generators $C \in (\Gamma^d)_0$ the degree 1.
    \end{par}
    \begin{par}
        We define a graded ideal $\mathrm{Rat}(\Gamma^d)$ on $Z(\Gamma^d)$ 
        generated by the polynomials
        \begin{align}
            \label{schnitt-chw1}
            C_1 \cdot \cdots \cdot C_k & \textrm{ for } \{C_1, \ldots, C_k\} \not \in
            (\Gamma^d)_S, \\
            \label{schnitt-chw2}
            \Big(\sum_{C' \in (\Gamma^d)_0} C'\Big) C_1 &, \\
            \label{schnitt-chw3}
            \sum_{\substack{C' \in (\Gamma^d)_0\\ \pr_i(C')=\pr_i(C_2)}} C_1 C_2 C'
            & \textrm{ for } i \in \{1, \ldots, d\} \textrm{ with } 
            \pr_i(C_1) \neq pr_i(C_2).
        \end{align}
        We call $\mathrm{Rat}(\Gamma^d)$ the ideal of cycles rational equivalent to zero.
    \end{par}
    \begin{par}
        The graded ring
        \[ \KC(\Gamma^d) := Z(\Gamma^d) / \mathrm{Rat}(\Gamma^d) \]
        is called \emph{combinatorial Chow ring}.
    \end{par}
\end{defn}
\begin{par}
    The combinatorial Chow ring has the following properties:
\end{par}
\begin{satz}
    \label{schnitt-comb-props}
    \begin{enumerate}
        \item
            There exists a morphism of $\IZ$-modules
            \[
                \ldeg_{\Gamma^d}: \KC(\Gamma^d) \to \IZ
            \]
            such that for each set $\{C_0, \ldots, C_d\}$ of $d+1$ distinct vertices of
            $\Gamma^d$ the degree $\ldeg_{\Gamma^d}(C_0 \cdot \ldots \cdot C_d) = 1$.
        \item
            The local degree map can be calculated locally:
            By \cref{schnitt-graph-zerleg} we associate 
            to each $\gamma=(\gamma_1, \ldots, \gamma_d)$ with 
            $\gamma_1, \ldots, \gamma_d \in \Gamma(X)^1$ 
            an embedding $i_\gamma: I^d \to \Gamma^d$. 
            This gives a covering of $\Gamma^d$ and the local degree
            satisfies
            \[
                \ldeg_{\Gamma^d}(\alpha) = \sum_{\gamma \in (\Gamma_1)^d}
                \ldeg_{I^d}(i_\gamma^* \alpha),
            \]
            where $i_\gamma^*: \KC(\Gamma^d) \to \KC(I^d)$ denotes the morphism
            given by functoriality of $\KC$.
    \end{enumerate}
\end{satz}
\begin{proof}
    \cite[4.4]{publ1}
\end{proof}
\begin{par}
    In the local situation $\Gamma=I$ we can describe the situation more precisely:
\end{par}
\begin{par}
    Let the vertices of $I = \Delta[1]$ be denoted by $C_0$ and $C_1$ 
    with the ordering $C_0 < C_1$.
    Then the vertices of $I^d$ can be described using vectors $v \in \IF_2^d$:
    For each vector $v=(v_1, \ldots, v_d)$ let $C_v \in (I^d)_0$ denote the vertex with
    $\pr_i(C_v) = C_{v_i}$.
    The set $\{C_v \mid v \in \IF_2^d\}$ is a generating set for $\KC(I^d)_\IQ$.
    Using fourier transforms we get yet another generating set:
    \[
        F_v := \sum_{w \in \IF_2^d}(-1)^{\angles{v,w}} C_w.
    \]
    It turns out that the generating set $\{F_v \mid v \in \IF_2^d\}$ of $\KC(I^d)$
    is appropriate for the further calculations.
    We first note the following isomorphism:
\end{par}

\begin{satz}
    \label{schnitt-lok-psi}
    Consider $\KC(I^d)$ with the generating sets $\{C_v \mid v \in \IF_2^d\}$ 
    and $\{F_v \mid v \in \IF_2^d\}$ as defined above.
    There is an isomorphism of graded rings
    \[ \psi: \KC(I^d) \xrightarrow{\sim} \KC(I^d), \]
    which is uniquely determined by
    $\psi(C_v) = C_{v+(1,\ldots, 1)}$.
    The equation
    \[ \psi(F_v) = (-1)^{\angles{v,(1,\ldots, 1)}} F_v. \]
    holds and $\psi$ is compatible with the local degree $\ldeg_{I^d}$.
\end{satz}
\begin{proof}
    \cite[Prop 4.30]{publ1}
\end{proof}
\begin{par}
    The combinatorial Chow ring can be compared with the Chow ring of $W$:
\end{par}
\begin{thm}[({\cite[Prop 4.14, Prop 4.23]{publ1}})]
    \label{schnitt-grad-vergleich}
    Let $d \in \IN$, $X$ be a proper regular strict semi-stable curve over $S$ with 
    a fixed ordering of $X_s^{(0)}$. Let $X$ be the Gross-Schoen desingularization of
    $X^d$ by \cref{desi-prodkomp}.
    Then there exists a morphism or graded rings
    \[
        \varphi: \KC(\RK(W)) \to \CH_{W_s}(W)_\IQ
    \]
    such that
    $\varphi$ is an isomorphism in degree $1$ and the equation
    \[
        \ldeg_W(\varphi(\alpha)) = \ldeg_{\RK(W)}(\alpha)
    \]
    holds for each $\alpha \in \KC(\RK(W))^{d+1}$.
\end{thm}
\begin{par}
    We may now formulate the intersection pairing in analytical terms
    using the $n$-th lattice approximation of the functions 
    $f_0, \ldots, f_d$.
\end{par}
\begin{prop}
    \label{limit-lin-int}
    Let $n \in \IN$ and
    $f_0, \ldots, f_d \in \lid{(\Gamma_n)^d}$. Then
    \begin{equation}
        \label{limit-lin-int-eq}
        \angles{f_0, \ldots, f_d}_{W,n} = 
        n^{2d} \sum_{v_0, \ldots, v_d \in \IF_2^d}
        \ldeg_{I^d}(\prod_{i=0}^d F_{v_i}) 
        \int_{\Gamma^d} 
        \prod_{i=0}^d \tilde\Delta^v_n(f_i)
    \end{equation}
    holds.
\end{prop}
\begin{proof}
    \mbox{}
    \paragraph{\textbf{Case 1: $n=1$.}}
    \begin{par}
        Denote by $\varphi_W: \KC(\Gamma(X)^d) \to \CH_{W_s}(W)$ the 
        canonical morphism of the combinatorial Chow ring to 
        $\CH_{W_s}(W)$ of \cref{schnitt-grad-vergleich}.
        By linearity we may assume $f_i = f_{D_i}^{\betr{\cdot}}$ 
        for Cartier divisors $D'_0, \ldots D'_d \in \cadiv_{W_s}(W)$.
        These divisors are in the image of $\varphi_W$ and we choose
        pre-images $D_0, \ldots D_d \in \KC(\Gamma(X)^d)$ such that
        $D'_i = \varphi_W(D_i)$.
        By \cref{schnitt-grad-vergleich} we may calculate the local degree
        in $\KC(\Gamma(X)^d)$ and get together with \cref{schnitt-comb-props}(2)
        \[ \ldeg_{W}(\varphi_W(D_0 \cdot\cdots\cdot D_d)) = 
            \sum_{\gamma \in \Gamma_1(X)^d} \ldeg_{I^d} \circ i_\gamma^*(D_0 \cdot\cdots\cdot D_d).
        \]
        Therefore both sides of \cref{limit-lin-int-eq} are additive on cubes and 
        it suffices to deal with the case $\RK(X)=I$.
    \end{par}
    \begin{par}
        Let $\Gamma(X)=I$. The centre of the standard cube 
        $|\Gamma(X)^d|=|I^d|=[0,1]^d$ 
        is denoted by $x_M:=1/2 (1,\ldots, 1)$.
        We describe the vertices of $I^d$ as usual by 
        $\{ C_v \mid v \in \IF_2^d\}$ 
        and study for each function $f \in \sms(I^d)$
        the associated divisor by \cref{metr-div-fkt}:
        \[ D_f := \phi_1^{-1}(f) = \sum_{v \in \IF_2^d} f_i(C_v) C_v. \]
        The coordinates of $C_v$ as point in $I^d$ are given by
        $x_M - 1/2(-1)^v$, 
        where $(-1)^v = ((-1)^{v_1}, \ldots, (-1)^{v_d})$ for $v = (v_1, \ldots, v_d)$.
        We have
        $x_M - 1/2(-1)^v = x_M + (1/2)(-1)^{v+(1,\ldots,1)}$
        and thus
        $f(C_v) = f_{x_M}^{1/2}(v+(1, \ldots, 1))$. 
        Using the morphism
        $\psi: \KC(I^d) \to \KC(I^d), C_v \mapsto C_{v+(1, \ldots, 1)}$ 
        from \cref{schnitt-lok-psi}
        we calculate:
        \begin{align*}
            D_f & = \sum_{v \in \IF_2^d} f_{x_M}^{1/2}(v+(1,\ldots, 1))C_v 
            = \sum_{v \in \IF_2^d} f_{x_M}^{1/2}(v) \psi(C_v) \\
            &= \psi\left(
            \sum_{v \in \IF_2^d} \frac{1}{2^d} f_{x_M}^{1/2}(v) C_v
            \right) \\
            &= \psi\left(\sum_{v \in \IF_2^d}
                \Big(\frac{1}{2^d} \sum_{w \in \IF_2^d} 
                (-1)^{\angles{v,w}} f_{x_M}^{1/2}(w)
                \Big)
                \Big(
                \sum_{w \in \IF_2^d}(-1)^{\angles{v,w}}C_w
                \Big)
            \right)\\
            & = \sum_{v \in \IF_2^d} \Delta_{1/2}^v(f)(x_M) \psi(F_v).
        \end{align*}
        Putting this into the definition of the intersection paring we get
        \begin{align*}
            \angles{f_0, \ldots, f_d} 
            & = \ldeg_{I^d}\left(\sum_{v_0, \ldots, v_d \in \IF_2^d}\left(
                    \prod_{i=0}^d \Delta^{v_i}_{1/2}(f)(x_M) \psi(F_{v_i})
            \right)\right) \\
            & = \sum_{v_0,\ldots, v_d \in \IF_2^d} 
            \ldeg_{I^d}(\prod_{i=0}^d F_{v_i})
            \prod_{i=0}^d \Delta^{v_i}_{1/2}(f_i)(x_M).
        \end{align*}
        We used hereby that the degree is invariant under $\psi$ (\cref{schnitt-lok-psi}).
        % XXX: Referenz richtig machen!
        By the definition of the lattice approximation $\tilde\Delta$ we finally get
        \[ \prod_{i=0}^d \Delta^{v_i}_{\frac{1}{2}}(f_i)(x_M) 
            = \int_{\Gamma_1^d} \prod_{i=0}^d \tilde\Delta^{v_i}_1(f_i), \]
        the claim in the case $n=1$.
    \end{par}
    \paragraph{\textbf{Case 2: $n>1$}}
    \begin{par}
        By \cref{limit-schnittpaar-bw-wohldef} we have
        \[ 
            \angles{f_0, \ldots, f_d}_{W,n} 
            = n^d \angles{f_0, \ldots, f_d}_{W_n,1} 
        \]
        and by \cref{limit-unt-int-eq} 
        for each function $f \in \sms(\Gamma^d)$
        \[ 
            n^d \int_{|\Gamma^d|} f = \int_{|\Gamma_n^d|} f
        \]
        holds.
        Using \cref{schnitt-delta-kommut} and the claim in the case $n=1$ 
        we conclude
        \begin{align*}
            \angles{f_0, \ldots, f_d}_{W,n} 
            & = n^d \angles{f_0, \ldots, f_d}_{W_n,1} \\
            & = n^d \sum_{v_0, \ldots, v_d \in \IF_2^d}\ldeg_{I^d}(\prod_{i=0}^d F_{v_i})
            \int_{\Gamma_n^d}\prod_{i=0}^d \tilde\Delta^{v_i}_1(f_i) \\
            & = n^{2d} \sum_{v_0, \ldots, v_d \in \IF_2^d}\ldeg_{I^d}(\prod_{i=0}^d F_{v_i})
            \int_{\Gamma^d}\prod_{i=0}^d \tilde\Delta^{v_i}_n(f_i).
        \end{align*}
    \end{par}
\end{proof}

\begin{par}
    This description of the intersection pairing is used to generalise %it
    onto a bigger set of functions. For this purpose we use the following
    approximation of continuous functions by piecewise affine functions.
\end{par}
\begin{defn}
    \label{limit-approx-def}
    Let $f \in \smd(\Gamma^d)$ and $n \in \IN$. 
    The function $f^{(n)} \in \lid(\Gamma_n^d)$, which is uniquely defined by
    $f^{(n)}(p) = f(p)$ for each $p \in (\gamma_n^d)_0$,
    is called the 
    \emph{$n$\nobreakdash-th standard approximation} of $f$.
\end{defn}
\begin{bem}
    \label{limit-approx-delta}
    To calculate $\tilde\Delta^v_n(f)$ we need only the values
    of $f$ 
    on the vertices of the $n$\nobreakdash-th subdivision.
    Thus we have for each $v \in \IF_2^d$
    \[ \tilde\Delta^v_n(f) = \tilde\Delta^v_n(f^{(n)}). \]
\end{bem}

\begin{par}
    We can now formulate a general convergence result 
    for the standard approximation of functions in 
    $\sms(\Gamma^d)$.
    As precondition for all results we need the 
    vanishing condition from \cite[4.7]{publ1}:
\end{par}
\begin{defn}
    \label{limit-konv-bed}
    Let $d \in \IN$. 
    We say $d$ satisfies the \emph{vanishing condition},
    if for each partition $\Part$
    of the set $\{1, \ldots, d\}$ and
    each $v_0, \ldots, v_d \in \IF_2^d$ 
    the relation
    $\sum_{i=0}^d \alpha(\Part,v_i) < d+|\Part|$ implies 
    the equation
    \begin{equation}
        \label{limit-allg-bed}
        \ldeg_{I^d}(\prod_{i=0}^d F_{v_i}) = 0
    \end{equation}
    in the combinatorial Chow ring $\KC(I^d)$.
\end{defn}
\begin{par}
    By explicite calculations of the intersection numbers 
    we have already shown in \cite[Cor 4.36]{publ1}:
\end{par}
\begin{lem}
    \label{limit-vaco-23}
    For $d=2$ and $d=3$ the vanishing condition \cref{limit-konv-bed} is satisfied.
\end{lem}
\begin{satz}
    \label{limit-allg}
    If $d \in \IN$ satisfies the vanishing condition \cref{limit-konv-bed},
    then for each of the functions
    $f_0, \ldots, f_d \in \smd(\Gamma^d)$ the limit
    \[ \angles{f_0, \ldots, f_d} := 
        \lim_{n \to \infty} \angles{f_0^{(n)}, \ldots, f_n^{(n)}}_{W,n} \]
    exists.
    It can be calculated by
    \[ 
        \angles{f_0, \ldots, f_d} = 
        \sum_{\Part \textrm{ Partition}}
        \frac{1}{2^{d+|\Part|}}
        \sum_{\substack{
                v_0, \ldots, v_d \in \IF_2^d,\\ 
                \sum \alpha(v_i, \Part) = d+|\Part|
            }}
        \ldeg_{I^d}(\prod_{i=0}^d F_{v_i})
        \int_{\Delta_{\Part}} \prod_{i=0}^d D^{v_i}_{\alpha(v_i,\Part)}(f_i).
    \]
\end{satz}
\begin{proof}
    \begin{par}
        We use the description of the intersection pairing from \cref{limit-lin-int},
        \begin{equation}
            \label{limit-allg-eq1}
            \angles{f_0^{(n)}, \ldots, f_d^{(n)}} = 
            n^{2d} \sum_{v_0, \ldots, v_d} \ldeg_{I^d}(\prod_{i=0}^d F_{v_i})
            \int_{\Gamma^d} \prod_{i=0}^d \tilde\Delta_n^{v_i}(f_i).
        \end{equation}
        Since the sum of the characteristic functions
        $\sum\limits_{\Part\textrm{ Partition}} \charfkt_{\{x\in|\Gamma|^d \mid \diag(\tilde
                x^{n})=\Part\}}$ 
        is the constant function $\charfkt_{\Gamma^d}$,
        we may split the integral of \cref{limit-allg-eq1} 
        into components along the different ``pixelated diagonals'' (\cref{pixelated})
        and apply \cref{limit-subint-lem}:
        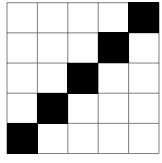
\begin{figure}[h]
            \begin{center}
                \begin{tikzpicture}[scale=2]
                    \draw[step=0.2 cm, help lines] (0,0) grid (1,1);
    %                \fill (0,0) rectangle (0.2,0.2);
                    \foreach \x in {0,0.2,0.4,0.6,0.8}
                      \fill[xshift=\x cm, yshift=\x cm] (0,0) rectangle (0.2,0.2);
                \end{tikzpicture}\\
                \caption{
                The ``pixelated diagonal'' $\{x \mid d(\tilde x^{(n)}) = \Part\}$ %\\ XXX
                with $n=5$ and $\Part=\{\{1,2\}\}$
                }
                \label{pixelated}
            \end{center}
        \end{figure}
        \begin{align*}
            n^{2d} \int_{\Gamma^d}\prod_{i=0}^d \tilde\Delta_n^{v_i}(f_i) 
            & = \sum_{\Part \textrm{ Partition}}
            n^{2d} \int_{\Gamma^d} \charfkt_{\{d(\tilde x^{(n)}) = \Part\}}
            \prod_{i=0}^d \tilde\Delta_n^{v_i}(f_i) \\
            & = \sum_{\Part \textrm{ Partition}} 
            n^{d+|\Part|}\int_{\diag_{\Part}} 
            \charfkt_{\{d(\tilde x^{(n)}) = \Part\}}
            \prod_{i=0}^d\tilde\Delta_n^{v_i}(f_i).
        \end{align*}
        Together with \cref{limit-allg-eq1} this implies
        \begin{gather*}
            \angles{f_0^{(n)}, \ldots, f_d^{(n)}} 
            = \sum_{\Part\textrm{ Partition}}\frac{1}{2^{d+|\Part|}}
            \sum_{v_0, \ldots, v_d \in \IF_2^d}
            \ldeg_{I^d}(\prod_{i=0}^d F_{v_i})
            \int_{\diag_\Part}
            T_n(\Part, v_0, \ldots, v_d) \\
            \intertext{with}
            T_n(\Part, v_0, \ldots, v_d) 
            := (2n)^{d+|\Part|}\charfkt_{\{d(\tilde x^{(n)})=\Part\}} \prod_{i=0}^d \tilde\Delta_n^{v_i}(f_i).
        \end{gather*}
        We study the convergence of the terms $T_n(\ldots)$:
        By the precondition \cref{limit-allg-bed} we only have to deal with terms where
        $\sum {\alpha(\Part,v_i)} \geq d+|\Part|$.
        By dominated convergence it suffices to show that all 
        $T_n(\Part, v_0, \ldots, v_d)$ are globally bounded and converge to 
        \[ T(\Part, v_0, \ldots, v_d):= \begin{cases}
                0 & \textrm{if } \sum \alpha(v_i, \Part) > d+|\Part|, \\
                \prod_{i=0}^d D^{v_i}_{\alpha(v_i, \Part)}(f_i) 
                & \textrm{if } \sum \alpha(v_i, \Part) = d+|\Part|.\\
            \end{cases}
        \]
    \end{par}
    \begin{par}
        For this purpose we rewrite $T_n$ as
        \[ T_n(\Part, v_0, \ldots, v_d) = 
            \charfkt_{\{d(\tilde x^{(n)})=\Part\}} 
            \cdot
            \Big(
            (2n)^{d+|\Part|-\sum \alpha(v_i,\Part)}
            \Big)
            \cdot
            \Big(
            \prod_{i=0}^d
            (2n)^{\alpha(v_i,\Part)}
            \tilde\Delta^{v_i}_n(f_i)
            \Big)
        \]
        and discuss each part individually.
        By \cref{limit-tdelta-konv} the function 
        $(2n)^{\alpha(v_i,\Part)} \Delta^{v_i}_n(f_i)$ is bounded
        and converges to
        $D^{v_i}_{\alpha(\Part,v_i)}$.
        The characteristic function 
        $\charfkt_{\{\Part(\tilde x^{(n)})=\Part\}}$
        is obviously bounded and converges to
        $\charfkt_{\{x \mid \diag(x) = \Part\}} = \charfkt_{\diag(\Part)}$.
        Finally the behaviour of the term 
        $n^{d+|\Part|-\sum\alpha(v_i,\Part)}$
        yields the convergence
        $\lim\limits_{n \to \infty} T_n(\Part, v_0, \ldots, v_d) = T(\Part, v_0, \ldots, v_d)$.
    \end{par}
\end{proof}

\begin{par}
    By \cref{limit-vaco-23} we are able to give a formulation of the intersection pairing
    without precodition for $d=2$ and $d=3$.
    The case $d=2$ yields the result of Zhang \cite[Prop 3.3.1, Prop 3.4.1]{zhang}:
\end{par}
\begin{kor}
    \label{limit-formel-zhang}
    Let $d=2$ and $f_0,f_1,f_2 \in \smd(\Gamma^2)$ continuous functions
    which are smooth on simplices. Let $f_i^{(n)}$ be the standard approximation
    of $f_i$.
    Then the limit of the triple pairing 
    $\lim_{n \to \infty} \angles{f^{(n)}_0, \ldots, f^{(n)}_2}$
    exists and can be calculated by
    \[
        \lim_{n \to \infty} \angles{f^{(n)}_0, \ldots, f^{(n)}_2}
        = \angles{f_0,f_1,f_2}_{\mathrm{sm}}
        + \angles{f_0,f_1,f_2}_{\mathrm{sing}}
    \]
    where 
    \begin{align*}
        \angles{f_0,f_1,f_2}_{\mathrm{sm}}
        &= \sum_{\substack{
                v_0,v_1,v_2 \in \IF_2^2 \\
                \{v_0,v_1,v_2\}=\{(1,0),(0,1),(1,1)\}
            }}
        \int_{|\Gamma^2|}
        D^{v_0}_{|v_0|}(f_0)
        D^{v_1}_{|v_1|}(f_1)
        D^{v_2}_{|v_2|}(f_2),
        \\
        \angles{f_0,f_1,f_2}_{\mathrm{sing}}
        &= \sum_{\substack{
                v_0,v_1,v_2 \in \IF_2^2 \\
                {\scriptscriptstyle\{v_0,v_1,v_2\}=\{(1,0),(0,1),(1,1)\}}
            }}
        2\int_{\diag}
        D^{v_0}_1(f_0)
        D^{v_1}_1(f_1)
        D^{v_2}_1(f_2)\\
        & \quad - 4\int_{\diag}
        D^{\scriptscriptstyle(1,1)}_1(f_0)
        D^{\scriptscriptstyle(1,1)}_1(f_1)
        D^{\scriptscriptstyle(1,1)}_1(f_2).
    \end{align*}
\end{kor}
\begin{proof}
    % XXX: better formulation
    \begin{par}
        Recall the intersection numbers calculated in \cite[Thm 4.32]{publ1},
        \[
            \ldeg(F_{v_1}F_{v_2} F_{v_3}) = \begin{cases}
                -32 & \textrm{ if } v_1=v_2=v_3 = (1,1), \\
                16  & \textrm{ if } \{v_1,v_2,v_3\} = \{(1,0),(0,1),(1,1)\}, \\
                0   & \textrm{ otherwise}.
            \end{cases}
        \]
    \end{par}
    \begin{par}
        The result is a direct consequence of \cref{limit-allg}:
        The summand with $\Part=\{\{1\},\{2\}\}$ yields the nonsingular part
        $\multiangles_{\mathrm{sm}}$. 
        Since for $v_0=v_1=v_2=(1,1)$ the equation
        \[ 
            \sum_{i=0}{2}\alpha(v_i,\Part) = \sum_{i=0}^2|v_i|= 6 > 4 = d+|\Part|
        \]
        holds, we only have to deal with elements $v_0,v_1,v_2 \in \IF_2^d$
        where $\{v_0,v_1,v_2\} = \{(1,0),(0,1),(1,1)\}$.
    \end{par}
    \begin{par}
        For the summand with $\Part=\{\{1,2\}\}$ we must take all 
        non-trivial intersections $F_{v_0}F_{v_1}F_{v_2}$ into account.
        The resulting term gives exactly the singular part 
        $\multiangles_{\mathrm{sing}}$ 
        in above formula.
    \end{par}
\end{proof}
\begin{bem}
    The exact formulation of \cite{zhang} is obtained 
    using \cref{limit-zhang-delta}, i.e., 
    by identifying
    \begin{align*}
         D_1^{(1,0)}(f)(x) &= \frac{\partial f}{\partial x_1}(x), \\
         D_1^{(0,1)}(f)(x) &= \frac{\partial f}{\partial x_2}(x), \\
         D_2^{(1,1)}(f)(x) &= \frac{\partial^2 f}{\partial x_1 \partial x_2}(x)
         \quad \textrm{if }x \not \in \diag, \\
         D_1^{(1,1)}(f)(x) &= \frac{1}{2} \frac{\partial}{\partial x_1}(f^+ - f^-)(x) =
         \frac{1}{2}\delta(f)(x)
         \quad \textrm{if }x \in \diag.
    \end{align*}
    Then the formula for the smooth resp. singular part
    becomes
    \begin{align*}
        \angles{f_0,f_1,f_2}_{\mathrm{sm}} 
        & = \int_{\Gamma^2 \setminus \diag} 
        \frac{\partial f_0}{\partial x_1}(x)
        \frac{\partial f_1}{\partial x_1}(x)
        \frac{\partial^2 f_2}{\partial x_1 \partial x_2}(x)
        \textrm{+ permutations}, \\
        \angles{f_0,f_1,f_2}_{\mathrm{sing}}
        & = \int_{\diag}\left(
            \frac{\partial f_0}{\partial x_1}(x)
            \frac{\partial f_1}{\partial x_2}(x)
            \delta(f_2)(x)
            \textrm{+ permutations} \right) \\
            & \quad 
            - \int_{\diag}\left(\frac{1}{2}\delta(f_0)(x)\delta(f_1)(x)\delta(f_2)(x)
        \right).
    \end{align*}
\end{bem}
\begin{par}
    A similar formula is deducible from \cref{limit-allg} in the case $d=3$. 
    For clarity reasons we only calculate the non-singular part of the pairing.
    This is achieved by calculating the intersection pairing
    only of functions $f \in \smq(\Gamma^3)$, i.e., functions smooth on cubes.
    For these functions the singular part vanishes.
\end{par}

\begin{satz}
    \label{limit-spez-3}
    Let $f_0, \ldots, f_3 \in \smq(\Gamma^3)$ be functions smooth on cubes.
    Then the limit of the quadruple pairing $\angles{f_0,f_1,f_2,f_3}$ 
    exists and can be calculated as
    \begin{align*}
        \lim_{n \to \infty} \angles{f_0^{(n)}, \ldots, f_3^{(n)}} 
        & = \int_{\Gamma^3} 
        \sum_{\substack{
                v_0,v_1,v_2,v_3 \in \IF_2 \\
                \{v_0,v_1,v_2,v_3\} \in B
            }}
        \prod_{i=0}^dD^{v_i}_{|v_i|}(f_i),
    \end{align*}
    where the set $B \subset \mathcal{P}(\IF_2^3)$ is defined as follows
    \begin{align*}
        B:= \Big\{
            &\{(1,0,0), (0,1,0), (0,0,1), (1,1,1)\}, \\
            &\{(1,0,0), (0,1,0), (1,0,1), (0,1,1)\}, \\
            &\{(1,0,0), (0,0,1), (1,1,0), (0,1,1)\}, \\
            &\{(0,1,0), (0,0,1), (1,1,0), (1,0,1)\} \Big\}.
    \end{align*}
\end{satz}
\begin{proof}
    \begin{par}
        Let $v_0, \ldots, v_3 \in \IF_2^3$.
        By \cite[Thm 4.33]{publ1} $\{v_0,v_1,v_2,v_3\} \in B$
        holds iff 
        $\ldeg_{\KC(I^3)}\left(\prod_{i=0}^3 F_{v_i}\right) \neq 0$
        and $\sum_{i=0}^3 |v_i| = 6$ hold.
        For these elements we have
        $\ldeg_{\KC(I^3)}(\left(\prod_{i=0}^3 F_{v_i}\right) = 2^6$
        and thus the term in \cref{limit-allg}
        belonging to $\Part=\{\{1\},\{2\},\{3\}\}$ is given by
        \begin{equation}
            \label{limit-expl-3-gl}
            \sum_{\substack{
                    v_0,v_1,v_2,v_3 \in \IF_2 \\
                    \{v_0,v_1,v_2,v_3\} \in B
                }}
            \prod_{i=0}^d D^{v_i}_{|v_i|}(f_i).
        \end{equation}
    \end{par}
    \begin{par}
        If $\Part$ is another partition of $\{1,2,3\}$, then there exists 
        for each $v_0, \ldots, v_3 \in \IF_2^3$ at least one $i \in \{0,1,2,3\}$
        such that
        $\alpha(v_i, \Part) < |v_i|$.
        This implies
        $D_{\alpha(v_i,\Part)}^{v_i}(f_i) = 0$ 
        since the functions are smooth.
    \end{par}
    \begin{par}
        Furthermore all functions are defined on $|\Gamma(X)^3|$ 
        and $|\Gamma(X)^3| \setminus \diag_{\{\{1\},\{2\},\{3\}\}}$
        is a zero-set, thus the claim is proven.
    \end{par}
\end{proof}
\begin{par}
    By \cref{limit-allg} one could think that the intersection pairing
    converges also for a broader set of functions. Then there is 
    a meaningful definition of positivity needed (cp. \cite{zsmpt}).
    Convergence without conditions can not be expected, 
    as the following example shows.
\end{par}
\begin{bsp}
    Let $X = \Proj{R[x_0,x_1,z]/ (x_0x_1 - z^2 \pi)}$ be the projective completion
    of the standard scheme $L$ with the usual ordering $<$ of the components of $L_s$
    and $W=W(X,<,2)$ the product model according to \cref{desi-prodkomp}.
    We identify as always $|\RK(W)| = |\Gamma(X)|^2 = [0,1]^2$
    and set
    \begin{align*}
        \varphi_n: [0,1] &\to \IR,\\
        x &\mapsto \frac{1}{2}\sum_{i=0}^n (-1)^i \max(0, \frac{1}{n} - |x-\frac{i}{n}|).
    \end{align*}
    The function $\varphi_n$ describes a triangle wave with amplitude $1$ and length
    $\frac{2}{n}$. 
    We define the following sequence of functions
    \begin{align*}
        f_{0,n} &:= \varphi_n(x), \\
        f_{1,n} &:= \varphi_n(y), \\
        f_{2,n} &:= \varphi_n(x-y).
    \end{align*}
    They are bounded and lie in $\lid(\unt_n(\Gamma))$. 
    We have however
    \[ 
        \angles{f_{0,n},f_{1,n},f_{2,n}} = n.
    \]
    By introducing a factor $f'_{i,n}:=n^{1/3}f_{i,n}$ we get functions
    which converge uniformly to $0$, but whose triple pairing is constant:
    \[ 
        \angles{f'_{0,n},f'_{1,n},f'_{2,n}} = 1.
    \]
\end{bsp}
\begin{proof}
    \begin{par}
        Obviously the differential of the function $\varphi_n$ satisfies
        \[ \varphi'_n(x) = \begin{cases}
                1  & \textrm{ if } x \in (\frac{2i}{n}, \frac{2i+1}{n}), \\
                -1 & \textrm{ if } x \in (\frac{2i+1}{n}, \frac{2i+2}{n}). 
            \end{cases}
        \]
        This gives for the generalised differentials
        \begin{align*}
            D^{(1,0)}_1(f_{0,n})(x,y) &= \varphi'_n(x), \\
            D^{(0,1)}_1(f_{1,n})(x,y) &= \varphi'_n(y), \\
            D^{(1,0)}_1(f_{2,n})(x,y) &= \varphi'_n(x-y), \\
            D^{(0,1)}_1(f_{2,n})(x,y) &= -\varphi'_n(x-y).
        \end{align*}
        Furthermore 
        \[ D^{(0,1)}_1(f_{0,n}) = D^{(1,0)}_1(f_{1,n}) = D^{(1,1)}_2(f_{i,n}) = 0. \]
        The diagonals of the $n$\nobreakdash-fold subdivision are given by the points $(x,y) \in
        [0,1]^2$ with $|x-y| = \frac{i}{n}$.
        At this point only $f_{2,n}$ has a singularity and we get
        \[ D^{(1,1)}_1(f_{2,n})(x,y) = (-1)^i \quad \text{for } |x-y| = \frac{i}{n}. \]
    \end{par}
    \begin{par}
        We may now apply the formula from \cref{limit-formel-zhang}:
        Since $D^{(1,1)}_2(f_{i,n}) = 0$ outside of the diagonal for each $i=0,1,2$,
        it suffices to calculate the singular part.
        This is given by
        \begin{align*}
            \angles{f_{0,n},f_{1,n},f_{2,n}} &= \int_{\diag}
            D^{(1,0)}_1(f_{0,n})D^{(0,1)}_1(f_{1,n})D^{(1,1)}_1(f_{2,n}) \\
            & = \int_{\diag} 1 = n.
        \end{align*}
    \end{par}
\end{proof}
% }}}

    \appendix
    % Appendix: "Simpliziale Komplexe"
% vi:ai:sts=4:sw=4:tw=90:foldmethod=marker:syntax=tex

\section{The geometric realization of simplicial sets and their subdivision}
\begin{par}
    As in \cite{publ1} we need some basic facts about simplicial sets. 
    To recall the definition and notation of 
    \emph{simplicial sets}, 
    the \emph{standard-$n$-simplex}, \emph{degenerate simplices}
    please see the appendix of \cite[Appendix A]{publ1}.
\end{par}
% \begin{defn}
%     \label{sk-standard}
%     The functor $\Delta[n]:=\Hom_{\Delta}(\cdot, [n])$ is a simplicial set,
%     the \emph{standard $n$-simplex}.
% \end{defn}
\begin{defn}
    \label{sk-einfach-kompl}
    Let $k \in \IN$.
    With $s_i$, for each $i \in \{0,\ldots,k\}$, the morphism
    \[ s_i: [0] \to [k], 0 \mapsto i \]
    is denoted.
    A simplicial set $\RK$ is called \emph{simplicial set without multiple simplices},
    if the map
    \[ \varphi: \coprod_{k=0}^\infty K^{\mathrm{nd}}_k \to \mathcal{P}(K_0), 
        t\in K^{\mathrm{nd}}_k \mapsto \{K(s_0)(t), \ldots, K(s_k)(t)\} \]
    is a monomorphism.
    If this is true, we denote the image of $\varphi$ by
    \[
        \RK_S := \im(\varphi) \subseteq \mathcal{P}(\RK_0).
    \]
\end{defn}
% \begin{prop}
%     \label{sk-kolim-simpl}
%     For each simplicial set $\RK_\cdot \in \sset$ 
%     the equations
%     \[ \RK_n \simeq \Hom_{\sset}(\Delta[n], \RK_\cdot) \]
%     and 
%     \[ \RK_\cdot = \colim_{\Delta \RK_\cdot} \Delta[n] = \colim_{\Delta' \RK_\cdot} \Delta[n].
%     \]
%     hold.
% \end{prop}
% \begin{proof}
%     A proof is given in \cite[Lemma 3.1.3, Lemma 3.1.4]{hovey}.
% \end{proof}
\begin{bem}
    \label{schnitt-graph-zerleg}
    Let $\Gamma$ be a graph without multiple simplices.
    By functoriality
    we identify the 1-simplices $\gamma_1 \in \Gamma_1$ with
    morphisms $i_{\gamma_1}: \Delta[1] \to \Gamma$.
    Since $\Gamma$ is without multiple simplices, the $i_{\gamma_1}$ are injective 
    for each non-degenerate 1-simplex $\gamma_1$.
    Let now $\gamma:=(\gamma_1, \ldots \gamma_d) \in (\Gamma_1^{\mathrm{nd}})^d$ 
    be a $d$-tuple of 1-simplices.
    The product 
    \[
        i_\gamma := (i_{\gamma_1} \times \cdots \times \cdots \times i_{\gamma_d}):
        I^d \to \Gamma^d
    \]
    is injective as well and denoted by $i_\gamma$. 
    By \cite[Prop 4.18]{publ1} the set of all $i_\gamma$ gives a covering of $\Gamma^d$.
\end{bem}
\begin{defn}
    Let $n \in \IN_0$.
    Then $|\Delta[n]|$ denotes the topological standard-$n$-simplex,
    i.e., the space
    \[ 
        \{(t_0, \ldots, t_n)\in \IR^{n+1} \mid \sum_{i}t_i = 1, t_i \geq 0 \}
        \subseteq \IR^{n+1}.
    \]
    Let $n, m \in \IN_0$ and $\varphi: [n] \to [m]$ be a
    morphism in the category $\Delta$. Then $\varphi$
    induces a continuous morphism
    \begin{align*}
        |\Delta[\varphi]|: |\Delta[n]| &\to |\Delta[m]|, \\
        (t_0, \ldots, t_n) &\mapsto (t'_0, \ldots t'_m)
    \end{align*}
    where $t'_j:= \sum_{\varphi(i)=j} t_i$.
    This makes $|\Delta[\cdot]|: \Delta \to \mathrm{Top}$ a covariant functor.
    For each simplicial set $\RK_\cdot$ we call the topological space
    \[
        |\RK| := \colim_{\Delta \RK_\cdot} |\Delta[n]|
    \]
    \emph{geometric realization} of $\RK_\cdot$.
    As colimit this construction is functorial in $\RK_\cdot$.
\end{defn}
\begin{par}
    For simplicial sets without multiple simplices we have the following
    more explicit description of the geometric realization:
\end{par}
\begin{prop}
    \label{sk-geo-real-koord}
    \begin{par}
        Let $\RK_\cdot$ be a simplicial set without multiple simplices.
        The geometric realization is the subspace
        $|K| \subseteq \hom_{\mathrm{set}}(\RK_0, \IR)$ consisting of the probability
        distributions on $\RK_0$ 
        (with respect to the counting measure)
        with support in a simplex of $\RK_\cdot$.
        An element of $|K|$ is therefore a function $f:\RK_0 \to [0,1]$
        with
        \begin{equation}
            \label{sk-geo-real-koord-eq}
            \sum_{v \in \RK_0}f(v) =1 \textrm{ and } 
            \supp(f) \in \RK_S 
            %\{K(\varphi_0)(t), \ldots, K(\varphi_n)(t)\}
            %\textrm{ f\"ur ein } t \in \RK_n
        \end{equation}
        with $\RK_S$ defined as in \cref{sk-einfach-kompl}.
    \end{par}
    \begin{par}
        If $\varphi: \RK_\cdot \to \RK'_\cdot$ is a morphism of simplicial sets
        without multiple simplices, the induced morphism
        $\varphi_*: |K| \to |K'|$
        is given on probability distributions as follows:
        Let $f \in |K|$ be as in \cref{sk-geo-real-koord-eq}.
        Then $f'=\varphi_*(f)$ is given by the map $f':\RK'_0 \to \IR$
        with
        \[ f'(s') = \sum_{s \in \varphi_0^{-1}(s')} f(s). \]
    \end{par}
\end{prop}
\begin{proof}
    Denote the set of probability distributions by $\WV(K)$. 
    For the standard simplices $\Delta[n]$ and morphisms of standard simplices 
    the isomorphism is obvious.
    We therefore get a continuous map $|K| \to \WV(K)$. It is easy to see
    that this map is open. Therefore it is enough to show that the map
    is also bijective.
    We calculate the inverse:
    Let $f: \RK_0 \to \IR$ be a probability distribution as in \cref{sk-geo-real-koord-eq}.
    Since the simplicial set $\RK_\cdot$ has no multiple simplices,
    there is an unique non-degenerate simplex $s: \Delta[j] \to \RK_\cdot$
    % XXX: nondegenerate n"otig?
    with $\supp(f) = \im(s)$.
    Then there is a morphism $f': \Delta[j]_0 \to \IR$ with $f=s_*(f')$ 
    and we map $f$ onto the point $|s|(f')$.
\end{proof}
\begin{bsp}
    For the standard-1-simplex $\Delta[1]$ we have $\Delta[1]_0 \simeq \{0,1\}$
    and $\Delta[1]_S \simeq \mathcal{P}(\{0,1\})$. Thus there is a canonical
    isomorphism between the geometric realization $|\Delta[1]|$ and the 
    interval $[0,1]$.
\end{bsp}
\begin{prop}
    \label{sk-simpl-perm}
    Let $I$ denote the standard-1-simplex $I:=\Delta[1]$
    and $S_d$ the symmetric group of degree $d$.
    Then there is a canonical bijection
    \[ \psi: S_d \to (I^d)^{\mathrm{nd}}_d \]
    between non-degenerate $d$-simplices in the product $I^d$
    and $S_d$.
    The geometric realization of the simplex $\psi(\sigma)$
    is given in $I^d = [0,1]^d$ by
    \[ 
        \{ x=(x_1, \ldots, x_d) \in [0,1]^d \mid 
            x_{\sigma(1)} \leq x_{\sigma(2)} \leq \cdots \leq x_{\sigma(d)} \} .
    \]
\end{prop}
\begin{proof}
    \begin{par}
        Since the product $I^d$ is defined component by component, we have
        \[
            I^d([n]) \simeq \prod_d \hom_{\poset}([n], [1])
            \simeq \hom_{\poset}([n],[1]^d),
        \]
        where the product $[1]^d$ is calculated in the category of sets with partial order
        (see \cite[Cor A.7]{publ1}). 
        To identify the $d$-simplices of $I^d$ we use
        \[
            (I^d)_d \simeq \hom([d], [1]^d).
        \]
        An element $\varphi \in \hom([d], [1]^d)$ is non-degenerate iff 
        $\varphi(0) < \varphi(1) < \cdots < \varphi(d)$ holds 
        with $<$ being the product ordering on $[1]^d$.
        Then there exists a unique permutation $\sigma \in S_d$ such that
        \[
            \varphi(i) = (0,\ldots,0,\underbrace{1,\ldots,1}_i)^\sigma
        \]
        holds for each $i \in \{0, \ldots, d\}$.
    \end{par}
    \begin{par}
        In the geometric realization $[0,1]^d$ the vertices of the simplex $\varphi$ 
        are given by the values of $\varphi(i)$.
        Each point of this simplex  
        is a convex combination of these points, 
        thus $x=(x_1, \ldots, x_d) \in |I^d|$ is in the simplex $\varphi$ iff
        \[
            x_{\sigma(1)} \leq \cdots \leq x_{\sigma(d)}
        \]
        holds.
    \end{par}
\end{proof}
\begin{par}
    For the description of ramified base-change we need a subdivision of simplicial sets.
    This can also described completely categorical:
\end{par}
\begin{defn}
    \mbox{}
    \label{sk-unt-def}
    \begin{enumerate}[(i)]
        \item
            Let $k \in \IN$. We denote by $\tilde \unt_k$ the functor
            \[
                \tilde\unt_k: \Delta \to \Delta 
            \]
            given on objects by
            \[ [n] \mapsto [(n+1)\cdot k - 1] \]
            and on morphisms by
            \[ \Hom_\sset([n],[m]) \ni \varphi \mapsto 
                \left( ak+b \mapsto ak+\varphi(b)
                    \textrm{ for } 0 \leq b < k
                \right).
            \]
        \item
            The functor induces by $\tilde\unt_k$
            \[ \unt_k: \sset \to \sset, \RK_\cdot \mapsto \tilde\unt_k \circ \RK_\cdot \]
            is called the \emph{$k$-fold subdivision functor}.
    \end{enumerate}
\end{defn}
\begin{prop}
    \label{sk-unt-kanon}
    For each simplicial set $X_\cdot$ there is a canonical isomorphism
    \[ \Unt_n: |\unt_n(X)| \simeq |X|. \]
    For $X_\cdot = I_\cdot = \Delta[1]$
    and with the description of the geometric realization
    of \cref{sk-geo-real-koord}
    this isomorphism is given by the mapping
    \[ \left[f: (\unt_n(I))_0 \to \IR\right] \mapsto \left[f': I_0 \to \IR \right]\]
    where
    \[ f'(0) = \sum_{i=0}^{n} \frac{1}{n} f(\varphi^k_i). \]
\end{prop}
\begin{proof}
    The construction of this mapping is given
    in \cite[Lemma 1.1]{boek}.
\end{proof}
\begin{prop}
    \label{sk-unt-kanon-einb}
    Let $\RK_\cdot$ be a simplicial set and
    $t: \Delta \to \unt_n(K)$ a simplex of $\unt_n(K)$.
    Then there is a simplex $s: \Delta \to \RK_\cdot$ in $\RK_\cdot$
    such that the image of $|t|$ under the canonical morphism
    \[ |\unt_n(K)| \simeq |K| \]
    lies completely in $\im(|s|)$.
\end{prop}
\begin{proof}
    % XXX: noch zu korrigieren
    Let $i \in \IN$ such, that $t \in (\unt_k(K))_i$ holds. 
    By the definition of the subdivision we have 
    $(\unt_k(K))_i \simeq \RK_{(i+1)k - 1}$ 
    and we choose $s \in \RK_{(i+1)k - 1}$ as image of $t$ 
    under this isomorphism.
    Then $t$ allows a factorisation of the form
    \[ 
        t: \Delta[i] \xrightarrow{\tilde t} \unt_k(\Delta[i]) \xrightarrow{\unt_k(s)} \unt_k(K) 
    \] 
    and therefore the following diagram commutes:
    \[
        \begin{CD}
            |\Delta[i]| @. \\
            @V|\tilde t| VV @. \\
            |\unt_k(\Delta[i])| @>\simeq >> |\Delta[(i+1)k-1]| \\
            @V |\unt_k(s)| VV @V |s| VV  \\
            |\unt_k(K)| @>\simeq>> |K|.
        \end{CD}
    \]
    This finishes the proof.
\end{proof}

    \bibliography{literatur}
\end{document}